\documentclass[11pt]{amsart}

\usepackage{graphicx,multirow,pifont}

\newtheorem{proposition}{Proposition}
\newtheorem{theorem}[proposition]{Theorem}
\newtheorem{lemma}[proposition]{Lemma}
\newtheorem{corollary}[proposition]{Corollary}

\theoremstyle{definition}
\newtheorem{definition}[proposition]{Definition}

\theoremstyle{remark}

\newtheorem{example}[proposition]{Example}

\numberwithin{proposition}{section}

\input xy
\xyoption{all}

\def\ur{\leq}
\def\ll{\geq}
\def\alg{\mathcal{A}}
\def\Ahat{\hat{\mathbf{A}}}
\def\Acheck{\check{\mathbf{A}}}
\def\ahat{\hat{a}}
\def\acheck{\check{a}}
\def\A{\mathbf{A}}
\def\B{\mathbf{B}}
\def\C{\mathbf{C}}
\def\D{\mathbf{D}}
\def\E{\mathbf{E}}
\def\F{\mathbf{F}}
\def\Lam{\mathbf{\Lambda}}
\def\Del{\mathbf{\Delta}}
\def\Th{\mathbf{\Theta}}
\def\I{\mathcal{I}}
\def\Z{\mathbb{Z}}
\def\R{\mathbb{R}}
\def\CC{\mathbb{C}}
\def\d{\partial}
\def\CThat{\widehat{\mathit{CT}}}
\def\CTminus{\mathit{CT}^{-}}
\def\CTtildeminus{\widetilde{\mathit{CT}}^{-}}
\def\CTdoublehat{\widehat{\widehat{\mathit{CT}}}}
\def\ddoublehat{\widehat{\widehat{\partial}}}
\def\dminus{\partial^-}
\def\dhat{\widehat{\partial}}
\def\dinfty{\partial^{\infty}}
\def\HThat{\widehat{\mathit{HT}}}
\def\HTdoublehat{\widehat{\widehat{\mathit{HT}}}}
\def\HTminus{\mathit{HT}^{-}}
\def\CTinfty{\mathit{CT}^{\infty}}
\def\HTinfty{\mathit{HT}^{\infty}}
\def\Aut{\operatorname{Aut}}
\def\Phil{\mathbf{\Phi}^L}
\def\Phir{\mathbf{\Phi}^R}
\def\ext{\text{ext}}
\def\Mat{\operatorname{Mat}}
\def\sl{\operatorname{sl}}
\def\diag{\operatorname{diag}}
\def\Aug{\operatorname{Aug}}
\def\op{\operatorname{op}}
\def\mir{\operatorname{mir}}
\def\kk{\mathbf{k}}
\def\tA{\tilde{\A}}
\def\tB{\tilde{\B}}

\def\tD{\tilde{\D}}

\def\tF{\tilde{\F}}
\def\ta{\tilde{a}}
\def\tb{\tilde{b}}
\def\td{\tilde{\partial}}
\def\tPhil{\tilde{\mathbf{\Phi}}^L}
\def\tPhir{\tilde{\mathbf{\Phi}}^R}
\def\tAU{\tilde{\A}^U}
\def\tAV{\tilde{\A}^V}
\def\tBU{\tilde{\B}^U}
\def\tBV{\tilde{\B}^V}

\newcommand{\rb}[1]{\raisebox{0.15in}{#1}}

\theoremstyle{plain}

\begin{document}

\title{Combinatorial Knot Contact Homology and
  Transverse Knots}
\author[L. Ng]{Lenhard Ng}
\address{Mathematics Department, Duke University, Durham, NC 27708}
\email{ng@math.duke.edu}
\urladdr{http://www.math.duke.edu/~{}ng/}

\begin{abstract}
We give a combinatorial treatment of transverse homology, a new
invariant of transverse knots that is an extension of knot contact
homology. The theory comes in several flavors, including one that is
an invariant of topological knots and produces a three-variable knot
polynomial related to the $A$-polynomial.
We provide a number of computations of transverse homology
that demonstrate its effectiveness in distinguishing transverse knots,
including knots that cannot be distinguished by the Heegaard Floer
transverse invariants or other previous invariants.
\end{abstract}

\maketitle

\section{Introduction and Results}
\label{sec:intro}

In this paper, we extend the combinatorial form for knot contact
homology introduced in \cite{NgKCH1,NgKCH2,Ngframed} to an invariant
of knots in $\R^3$ that are transverse to the standard
contact structure $\ker(dz-y\,dx)$. This new invariant, which we call \textit{transverse homology}, comes in several flavors. One version, \textit{infinity
  transverse homology}, can be seen as an invariant of topological
knots in $\R^3$ with no reference to a contact structure; as such, it
is an extension and generalization of knot contact homology.

The transverse invariant described here is a combinatorial form of a
more general invariant of transverse knots in contact
$3$-manifolds. This general invariant is introduced in
\cite{EENStransverse}, where it is also proven to specialize to the
combinatorial complex studied in this paper in the case of knots in standard
contact $\R^3$.\footnote{The
combinatorial invariant defined in \cite{EENStransverse} uses slightly
different sign conventions from the
one defined here, but the two invariants are equivalent. See
Section~\ref{ssec:eens} for further discussion.} Here is a brief summary of the geometric construction in
\cite{EENStransverse}, building on the work in \cite{EENS}, in the
case of interest to us: a
knot $K \subset (\R^3,\xi_{\text{std}})$ lifts
to a Legendrian torus $\Lambda_K$, the conormal bundle to $K$, in the
cosphere bundle
$ST^*\R^3$; the contact structure $\xi$ similarly
lifts to two submanifolds $\widetilde{\xi},\widetilde{-\xi}\subset
ST^*\R^3$ corresponding to the two possible coorientations of
$\xi$. If $K$ is transverse to $\xi$, then $\Lambda_K$ is disjoint
from both $\widetilde{\xi}$ and $\widetilde{-\xi}$. Now
$\widetilde{\xi}$ and $\widetilde{-\xi}$ lift to $4$-dimensional
submanifolds in the symplectization of $ST^*\R^3$, and these submanifolds are holomorphic with respect to an appropriate almost complex structure. Counting intersections of holomorphic disks with these holomorphic $4$-manifolds induces filtrations on the Legendrian contact homology
complex for $\Lambda_K$, resulting in the invariant that we call
transverse homology. We keep track of these two
filtrations through two parameters $U$ and $V$, which encode intersection numbers with $\widetilde{\xi}$ and $\widetilde{-\xi}$.

We will henceforth suppress the geometric origins of transverse
homology, and treat the invariant purely combinatorially. In
particular, we will present combinatorial proofs of the invariance
results for transverse homology. Please see \cite{EENStransverse} for
the parallel contact-geometric story.

The formulation of transverse homology described in this paper is
defined in terms of a braid whose closure is the topological or
transverse knot under consideration. By the classical Markov Theorem,
isotopy classes of topological knots (or links) correspond to
equivalence classes of braids modulo conjugation and positive/negative
stabilization/destabilization. For the purposes of this paper, we can
define a transverse knot likewise.

\begin{definition}
A \textit{transverse knot} or link in $\R^3$ is an equivalence class
of braids modulo conjugation and positive de/stabilization.
\end{definition}

\noindent
This definition is equivalent to the usual definition, a knot
everywhere transverse to the standard contact structure, by the
Transverse Markov Theorem \cite{OSh,Wrinkle}.

We now present our main results. Throughout this paper, we write $R =
\Z[\lambda^{\pm
  1},\mu^{\pm 1}]$. To any braid $B$ whose closure has one
component\footnote{All of the invariants we present in this paper have
  corresponding analogues for multi-component links, where the base
  ring $R$ is replaced by $\Z[\lambda_1^{\pm 1},\mu_1^{\pm
    1},\ldots,\lambda_k^{\pm 1},\mu_k^{\pm 1}]$ for $k$ the number of
  components of the link. We have specialized to the one-component
  knot case for simplicity; see \cite{EENStransverse} for more on the
  general link invariants.}, we will associate
a differential graded algebra over $R[U,V]$, the
\textit{transverse complex}
\[
B \leadsto (\CTminus_*(B),\dminus(B)).
\]
We remark that unlike in Heegaard Floer homology, $U$ and $V$ have
grading $0$. Define the \textit{transverse homology} of $B$ to be the
graded $R[U,V]$-algebra $\HTminus_*(B) = H_*(\CTminus(B),\dminus(B))$.

\begin{theorem}
Up to stable tame isomorphism, the transverse complex
$(\CTminus_*(B),\dminus(B))$ is invariant under braid conjugation and
positive braid stabilization, and thus constitutes an invariant of the
transverse knot underlying the braid $B$. In particular, $\HTminus_*(B)$ is a transverse invariant. In addition, setting $U=V=1$ in
$(\CTminus_*(B),\dminus(B))$ yields the framed knot DGA from
\cite{Ngframed}.
\label{thm:main}
\end{theorem}

As is familiar from Heegaard Floer homology or similar settings, we
can derive several
auxiliary complexes from the transverse complex.

\begin{definition}
Let $B$ be a braid.
\begin{itemize}
\item
The \textit{hat transverse complex} $(\CThat_*(B),\dhat(B))$ is the
differential graded algebra over $R$ obtained from
$(\CTminus_*(B),\dminus(B))$ by setting $U=0$ and $V=1$; the
\textit{hat transverse homology} $\HThat_*(B)$ is the graded
$R$-algebra $H_*(\CThat(B),\dhat(B))$.
\item
The \textit{double-hat transverse complex}
$(\CTdoublehat_*(B),\ddoublehat(B))$ is the differential graded
algebra over $R$ obtained from $(\CTminus_*(B),\dminus(B))$ by setting
$U=V=0$; the
\textit{double-hat transverse homology} $\HTdoublehat_*(B)$ is the graded
$R$-algebra $H_*(\CTdoublehat(B),\ddoublehat(B))$.
\item
The \textit{infinity transverse complex}
$(\CTinfty_*(B),\dinfty(B))$ is the differential graded algebra
over $R[U^{\pm 1},V^{\pm 1}]$ obtained from
$(\CTminus_*(B),\dminus(B))$ by tensoring with $R[U^{\pm 1},V^{\pm
  1}]$ and replacing $\lambda$ by $\lambda (U/V)^{-(\sl(B)+1)/2}$,
where $\sl(B) = w(B)-n(B)$ is the difference between the writhe
(algebraic crossing number) of $B$ and the number of strands in $B$;
the
\textit{infinity transverse homology} $\HTinfty_*(B)$ is the graded
$R[U^{\pm 1},V^{\pm
  1}]$-algebra $H_*(\CTinfty(B),\dinfty(B))$.
\end{itemize}
\label{def:flavors}
\end{definition}

\noindent
Note the substitution $\lambda \mapsto \lambda (U/V)^{-(\sl(B)+1)/2}$
in the infinity theory, which is allowed
once $U,V$ are invertible, and which is necessary for the invariance
statement below.

Besides the complexes listed in
Definition~\ref{def:flavors} above, there are many others that one
might consider: e.g., the DGA over $R$ obtained by setting $U=1,
V=0$. Many of these are equivalent to one of the complexes from
Definition~\ref{def:flavors}; see the discussion of symmetries in
Section~\ref{sec:symmetries}. There are other variants, e.g., a ``$+$
version'', that
we will not consider here.

\begin{theorem}
Let $B$ be a braid.
\begin{enumerate}
\item
\label{infty1}
Up to stable tame isomorphism, the DGAs $(\CThat_*(B),\dhat(B))$ and
$(\CTdoublehat_*(B),\ddoublehat(B))$ are invariants of the transverse
knot underlying $B$. In particular, $\HThat_*(B)$ and
$\HTdoublehat_*(B)$ are transverse invariants.
\item
\label{infty2}
Up to stable tame isomorphism, the DGA $(\CTinfty_*(B),\dinfty(B))$ is
an invariant of the knot type given by the closure of $B$. In
particular, $\HTinfty_*(B)$ is a topological invariant.
\end{enumerate}
\label{thm:infty}
\end{theorem}

Statement (\ref{infty1}) in Theorem~\ref{thm:infty} is a direct
corollary of Theorem~\ref{thm:main}, while the proof of statement
(\ref{infty2}) requires checking that $(\CTinfty_*(B),\dinfty(B))$ is
invariant under negative braid stabilization as well as conjugation
and positive stabilization.

We summarize the various flavors of transverse homology in the diagram
below. We indicate which of the flavors are transverse knot invariants and
which are topological knot invariants by using inputs denoted $T$ for a
transverse knot and $K$ for a topological knot. The final DGA
$CC_*(K)$ in the diagram is the framed knot DGA from \cite{Ngframed}.

\[
\fbox{
\xymatrix{
& \CThat_*(T) \textrm{ over } R  \\
\CTminus_*(T) \textrm{ over } R[U,V]
\ar[ur]^-{U=0,\,V=1 \hspace{5ex}} \ar[r]^-{U=V=0}
\ar[dr]^(0.6){\hspace{13ex} \otimes
  R[U^{\pm 1},V^{\pm
    1}];~~~
\lambda \mapsto \lambda(U/V)^{-(\sl+1)/2}} & \CTdoublehat_*(T) \textrm{ over } R  \\
& \CTinfty_*(K) \textrm{ over } R[U^{\pm 1},V^{\pm 1}] \ar[d]^{U=V=1} \\
& CC_*(K) \textrm{ over } R
}}
\]

\vspace{12pt}

We will see that
transverse homology
constitutes an \textit{effective transverse invariant}; that is, it
can be used to distinguish pairs of transverse knots with the same
topological type and self-linking number. In particular, we can use
transverse homology to produce the following result.

\begin{proposition}
The knot types $m(7_2)$, $m(7_6)$, $m(9_{44})$, $9_{48}$,
$m(10_{132})$, $10_{136}$, $m(10_{140})$, $m(10_{145})$, $m(10_{161})$, and $12n_{591}$ are \label{prop:nonsimple}
transversely nonsimple.
\end{proposition}

Six of the ten knots in Proposition~\ref{prop:nonsimple} were previously known to be nonsimple, using a Heegaard Floer transverse invariant (either grid-diagram \cite{OST} or LOSS \cite{LOSS}). The other four, $m(7_6)$, $m(9_{44})$, $9_{48}$, and $10_{136}$, cannot be
distinguished by any previously known invariant, including Heegaard
Floer.

As a crude gauge of the effectiveness of our invariant, there are
thirteen total knots with arc index at most $9$ that are conjectured
to be transversely nonsimple by the ``Legendrian knot atlas''
\cite{Atlas}; of these, transverse homology can prove nonsimplicity for 
at least the ten listed in Proposition~\ref{prop:nonsimple}.
Transverse homology
could in theory be applied to the three remaining knots as well,
though probably with some difficulty. See
Section~\ref{sec:applications} for further discussion, including a
table of all thirteen knots with transverse representatives in
grid-diagram and braid form.

The proof of Proposition~\ref{prop:nonsimple} via transverse homology
involves auxiliary calculations in
\textit{Mathematica}, using a program \texttt{transverse.m}
that computes the transverse DGA of any braid and can compute various
numerical invariants derived from the transverse DGA. Readers
interested in doing computations with transverse homology can download
\texttt{transverse.m} from the author's web site.

All of the applications to transverse nonsimplicity given in this paper use only a small part of the transverse homology invariant, the degree $0$ invariant $\HThat_0$ in the hat version. A reader familiar with knot contact homology may recall that the degree $0$ part of knot contact homology, $\mathit{HC}_0$, has a simple topological interpretation as the ``cord algebra'' of the knot \cite{NgKCH2,Ngframed}. A similar interpretation of degree $0$ transverse homology would be interesting, and in particular might lead to transverse nonsimplicity results for general families of knots, but does not presently exist. We hope to return to this issue in the future.

The topological version of transverse homology, $\HTinfty_*$, may be
of interest in its own right as a generalization of knot contact
homology. On the contact-geometric side, passing from knot contact
homology $\mathit{HC}_*(K)$ to infinity transverse homology
$\HTinfty_*(K)$ represents an extension of the base ring for the
Legendrian contact homology of $\Lambda K$ in the contact manifold
$ST^*\R^3$ from $\Z[H_1(\Lambda K)]$ to $\Z[H_2(ST^*\R^3,\Lambda
K)]$. See \cite{EENStransverse} for further discussion.

We will not discuss topological aspects of infinity transverse
homology much in this paper, but one can produce from $\HTinfty_0(K)$ a
three-variable knot polynomial $\Aug_K(\lambda,\mu,U)$, the
``augmentation polynomial'', in an analogous manner to the
two-variable augmentation polynomial $\tilde{A}_K(\lambda,\mu)$
produced from $\mathit{HC}_0(K)$ in \cite{Ngframed}. It was shown in
\cite{Ngframed} that $\tilde{A}_K(\lambda,\mu)$ contains the familiar
$A$-polynomial of $K$ as a factor. We suspect that
$\Aug_K(\lambda,\mu,1) = \tilde{A}_K(\lambda,\mu)$ in general, in
which case $\Aug_K(\lambda,\mu,U)$ could be viewed as some sort of
three-variable generalization of the $A$-polynomial.

Here is an outline of the rest of the paper. We present algebraic
definitions for the various flavors of transverse homology in
Section~\ref{sec:defs} and prove invariance
in Section~\ref{sec:proofs}. In Section~\ref{sec:symmetries}, we
establish some basic properties and symmetries of transverse
homology. Section~\ref{sec:applications} presents a number of
computations of transverse homology, discussing the three-variable
augmentation polynomial for topological knots, and demonstrating the
effectiveness of transverse homology as an invariant of transverse
knots.

\section*{Acknowledgments}

Much of the work for this paper was done during the author's residency at the
Mathematical Sciences Research Institute in spring 2010. In
particular, the starting point for this paper was a series of
conversations at MSRI between John Etnyre, Michael Sullivan, and the
author. I am deeply indebted to them and Tobias Ekholm, my collaborators on
the geometric side of this project \cite{EENStransverse}. I also thank
the referee for helpful comments. This work
was supported by NSF grant DMS-0706777 and NSF CAREER grant
DMS-0846346.

\section{The Invariants}
\label{sec:defs}

Let $B \in B_n$ be an $n$-strand braid. In this section, we associate
a differential graded algebra $(\alg,\d)$ to $B$, the transverse DGA
of $B$. Algebraically, the transverse DGA is just
an easy enhancement of the framed knot DGA from
\cite{Ngframed}. A main result of this paper is that, up to stable
tame isomorphism, $(\alg,\d)$ is invariant under braid conjugation and
positive braid stabilization, and is thus an invariant of the underlying transverse knot.

Recall that $R$ represents the ring $\Z[\lambda^{\pm
  1},\mu^{\pm 1}]$.
Write $\alg_n$ for the tensor algebra over $R[U,V]$ generated by
$n(n-1)$ formal variables $a_{ij}$ with $1\leq i,j\leq n$ and $i\neq
j$. As in \cite{Ngframed}, there is a representation $\phi
:\thinspace B_n \to \Aut \alg_n$ defined on the generators $\sigma_k$ of
$B_n$ by
\[
\phi_{\sigma_k} :\thinspace
\begin{cases}
a_{ki} \mapsto -a_{k+1,i}-a_{k+1,k}a_{ki}, & i\neq k,k+1 \\
a_{ik} \mapsto -a_{i,k+1}-a_{ik}a_{k,k+1}, & i\neq k,k+1 \\
a_{k+1,i} \mapsto a_{ki}, & i\neq k,k+1 \\
a_{i,k+1} \mapsto a_{ik}, & i\neq k,k+1 \\
a_{k,k+1} \mapsto a_{k+1,k} & \\
a_{k+1,k} \mapsto a_{k,k+1} & \\
a_{ij} \mapsto a_{ij}, & i,j \neq k,k+1.
\end{cases}
\]
We will write $\phi_B$ for the automorphism of $\alg_n$ assigned to a
braid $B$ by the map $\phi$.

It will be convenient to consider $n\times n$ matrices with entries in
$\alg_n$; write $\Mat_n(\alg_n)$ for the
$R[U,V]$-algebra of such matrices. To a braid $B\in B_n$, the homomorphism $\phi$ associates two matrices $\Phil_B,\Phir_B\in\Mat_n(\alg_n)$, as
follows. View $B$ as a braid in $B_{n+1}$ by adding an additional
strand labeled $n+1$, and write $\phi_B^{\ext}$ for the resulting
automorphism of $\alg_{n+1}$; then for any $i=1,\ldots,n$, we can write
\begin{align*}
\phi_B^\ext(a_{i,n+1}) &= \sum_{\ell=1}^n (\Phil_B)_{i\ell} a_{\ell,n+1} \\
\phi_B^\ext(a_{n+1,i}) &= \sum_{\ell=1}^n a_{n+1,\ell} (\Phir_B)_{\ell
  i}
\end{align*}
for $(\Phil_B)_{i\ell},(\Phir_B)_{\ell i}\in\alg_n$.

Assemble the
generators $a_{ij}$ into two matrices
$\A_\ll,\A_\ur\in\Mat_n(\alg_n)$, with
\[
(\A_\ll)_{ij} = \begin{cases} a_{ij} & i>j \\
-1 & i=j \\
0 & i<j
\end{cases}
\hspace{0.5in} \text{and}
\hspace{0.5in}
(\A_\ur)_{ij} = \begin{cases} 0 & i>j \\
-1 & i=j \\
a_{ij} & i<j.
\end{cases}
\]
Define two more matrices $\Ahat$, $\Acheck\in\Mat_n(\alg_n)$ by
\begin{align*}
\Ahat &= \A_\ll + \mu U \A_\ur \\
\Acheck &= V \A_\ll + \mu \A_\ur;
\end{align*}
that is,
\[
(\Ahat)_{ij} = \begin{cases} a_{ij} & i>j \\
-1-\mu U & i=j \\
\mu U a_{ij} & i<j
\end{cases}
\hspace{0.5in} \text{and}
\hspace{0.5in}
(\Acheck)_{ij} = \begin{cases} V a_{ij} & i>j \\
-V-\mu & i=j \\
\mu a_{ij} & i<j.
\end{cases}
\]
Note that when $U=V=1$, the matrices $\Ahat$ and $\Acheck$ are
identical and equal to the matrix
labeled $A$ in \cite{Ngframed}.

Finally, let
$\Lam_B,\Lam_B'$ denote the $n\times n$ diagonal matrices
$\diag(\lambda \mu^{-w(B)},1,1,\ldots,1)$ and
$\diag(\lambda \mu^{-w(B)}(U/V)^{-(\sl(B)+1)/2},1,1,\ldots,1)$, where $w(B),\sl(B)$ are the
writhe and self-linking number of $B$.

\begin{definition}
\begin{enumerate}
\item
The \textit{degree-$0$ transverse homology} of $B$, written
$\HTminus_0(B)$, is the $R[U,V]$-algebra given by the quotient
\[
\HTminus_0(B) = \alg_n \, / \, (\Ahat -
\Lam_B\cdot\Phil_B\cdot\Acheck,
\Acheck - \Ahat \cdot \Phir_B \cdot \Lam_B^{-1}),
\]
where $(\mathbf{M}_1,\mathbf{M}_2)$ represents the two-sided ideal in
$\alg_n\otimes
R[U,V]$ generated by the entries of the two matrices $\mathbf{M}_1$
and $\mathbf{M}_2$.

\item
The \textit{degree-$0$ hat transverse homology} and \textit{degree-$0$
  double hat transverse homology} of $B$ are the $R$-algebras given by
\begin{align*}
\HThat_0(B) &= \alg_n \, / \, (\Ahat -
\Lam_B\cdot\Phil_B\cdot\Acheck,
\Acheck - \Ahat \cdot \Phir_B \cdot \Lam_B^{-1})|_{U=0,V=1} \\
\HTdoublehat_0(B) &= \alg_n \, / \, (\Ahat -
\Lam_B\cdot\Phil_B\cdot\Acheck,
\Acheck - \Ahat \cdot \Phir_B \cdot \Lam_B^{-1})|_{U=V=0}.
\end{align*}

\item
The \textit{degree-$0$ infinity transverse homology} of $B$ is the
$R[U^{\pm 1},V^{\pm 1}]$-algebra given by
\[
\HTinfty_0(B) = (\alg_n\otimes R[U^{\pm 1},V^{\pm 1}])\, / \, (\Ahat -
\Lam_B'\cdot\Phil_B\cdot\Acheck,
\Acheck - \Ahat \cdot \Phir_B \cdot (\Lam_B')^{-1}).
\]
\end{enumerate}
\label{def:deg0}
\end{definition}

\noindent
Note that if we set $U=V=1$ in $\HTminus_0(B)$, we recover the
degree-$0$ framed knot contact homology (or ``cord algebra'')
$HC_0(B)$ from \cite{Ngframed}.

In practice, the degree-$0$ transverse homologies defined above are
easier to work with than the full transverse complex defined below. In
particular, the proofs that the degree-$0$ transverse homologies are
transverse invariants (or topological invariants, in the case of
$\HTinfty_0$) are easier than the invariance proofs for the full
transverse DGA, and the applications in Section~\ref{sec:applications}
rely only on $\HThat_0$.

We next construct the full invariant, a differential graded algebra
$(\CTminus_*(B),\dminus(B))$
whose homology in degree $0$ is $\HTminus_0(B)$. The
algebra $\CTminus_*(B)$ is the tensor algebra over $R[U,V]$ freely
generated by the following:
\begin{itemize}
\item
the $n(n-1)$ generators $a_{ij}$, where $1\leq i,j\leq n$ and $i\neq
j$, of degree $0$
\item
$n(n-1)$ generators $b_{ij}$, where $1\leq i,j\leq n$ and $i\neq j$,
of degree $1$
\item
$2n^2$ generators $c_{ij}$ and $d_{ij}$, where $1\leq i,j\leq n$, of
degree $1$
\item
$2n^2$ generators $e_{ij}$ and $f_{ij}$, where $1\leq i,j\leq n$, of
degree $2$.
\end{itemize}

Assemble these six families of generators into six $n\times n$
matrices $\A,\B,\C,\D,\E,\F$, where $\A=(a_{ij})$ and so forth, and the
diagonal entries in $\A,\B$ are all $-2,0$ respectively; also define
two auxiliary matrices $\hat{\B}$, $\check{\B}$ by
\[
(\hat{\B})_{ij} = \begin{cases} b_{ij} & i>j \\
0 & i=j \\
\mu U b_{ij} & i<j
\end{cases}
\hspace{0.5in} \text{and}
\hspace{0.5in}
(\check{\B})_{ij} = \begin{cases} V b_{ij} & i>j \\
0 & i=j \\
\mu b_{ij} & i<j,
\end{cases}
\]
cf.\ the definitions of $\hat{\A}$ and $\check{\A}$ (which reduce to
$\A$ if we set $U=V=\mu=1$).

One final piece of notation: if $\mathbf{M} \in \Mat_n(\alg_n)$,
then write $\phi_B(\mathbf{M})$ and $\dminus(\mathbf{M})$ for the
matrices whose entries are the images of the entries of $M$ under
$\phi_B$ and $\dminus$.

\begin{definition}
The \textit{transverse DGA} of $B$, written
$(\CTminus_*(B),\dminus(B))$, is the differential graded algebra over
$R[U,V]$ with $\CTminus_*(B)$ defined as above and the differential
$\dminus(B)$ given by
\label{def:transDGA}
\begin{align*}
\dminus(\A) &= 0 \\
\dminus(\B) &= \A - \Lam_B \cdot \phi_B(\A) \cdot \Lam_B^{-1} \\
\dminus(\C) &= \Ahat - \Lam_B \cdot \Phil_B \cdot \Acheck \\
\dminus(\D) &= \Acheck - \Ahat \cdot \Phir_B \cdot \Lam_B^{-1} \\
\dminus(\E) &= \hat{\B} - \C - \Lam_B \cdot \Phil_B \cdot \D \\
\dminus(\F) &= \check{\B} - \D - \C \cdot \Phir_B \cdot \Lam_B^{-1}.
\end{align*}
\end{definition}

As described in Section~\ref{sec:intro}, we can construct other
flavors of transverse complexes by applying various algebraic
operations to the transverse DGA:
\begin{align*}
(\CThat_*(B),\dhat(B)) &= (\CTminus_*(B),\dminus(B))|_{U=0,V=1} \\
(\CTdoublehat_*(B),\ddoublehat(B)) &=
(\CTminus_*(B),\dminus(B))|_{U=0,V=0} \\
(\CTinfty_*(B),\dinfty(B)) &= (\CTminus_*(B)\otimes R[U^{\pm 1},V^{\pm
  1}],\dminus(B))~~~(\lambda \mapsto \lambda (U/V)^{-(\sl(B)+1)/2}).
\end{align*}
The first two complexes are DGAs over $R$, while the last is a DGA
over $R[U^{\pm 1},V^{\pm 1}]$.

\begin{example}
\label{ex:unknot}
Let $T_U$ and $S(T_U)$ denote the standard $\sl=-1$ transverse unknot and
the $\sl=-3$ transverse unknot, respectively, so that $S(T_U)$ is the
transverse stabilization of $T_U$. Note that $T_U$ and $S(T_U)$ are the
closures of the trivial braid in $B_1$ and $\sigma_1^{-1} \in B_2$,
respectively. The transverse DGA for $T_U$ is generated by four
generators $c,d,e,f$, with differential
\begin{align*}
\dminus(c) &= -1-\mu U+\lambda V+\lambda\mu \\
\dminus(d) &= -\lambda^{-1}(-1-\mu U+\lambda V+\lambda\mu) \\
\dminus(e) &= -c-\lambda d \\
\dminus(f) &= -d-\lambda^{-1} c.
\end{align*}
Up to stable tame isomorphism, the generators $d,f$ in this DGA can be
eliminated, resulting in the DGA generated by $c,e$ with differential
\begin{align*}
\dminus(c) &= -1-\mu U+\lambda V+\lambda\mu \\
\dminus(e) &= 0;
\end{align*}
see Proposition~\ref{prop:modifiedDGA} for a more general result.
The degree-$0$ transverse homology of $T_U$ is
\[
\HTminus_0(U) = R[U,V]/(-1-\mu U+\lambda V+\lambda\mu),
\]
with corresponding results
for the other flavors of transverse homology.

For $S(T_U)$, the transverse DGA is more involved to write down
explicitly, and we will not do it here. However, it is straightforward
to calculate that the degree-$0$ transverse homology is
\[
\HTminus_0(S(T_U)) = (R[U,V])[a_{12}]/(1+\mu U+V a_{12},
V+\mu+U a_{12}/\lambda).
\]
On replacing $\lambda$ by $\lambda U/V$, we deduce that
$\HTinfty_0(S(T_U)) \cong R[U^{\pm 1},V^{\pm
  1}]/(-1-\mu U+\lambda
V+\lambda \mu) \cong \HTinfty_0(T_U)$.
\end{example}

The main invariance results about the transverse complexes,
Theorems~\ref{thm:main} and~\ref{thm:infty}, state that, up to equivalence,
$\CTminus_*(B),\CThat_*(B),\CTdoublehat_*(B)$ are invariants of
transverse knots, while $\CTinfty_*(B)$ is an invariant of topological
knots; we will prove these results in Section~\ref{sec:proofs}.
The precise notion of equivalence is stable tame isomorphism,
originally defined by Chekanov \cite{Ch02} for DGAs over $\Z/2$. Since
our equivalences are for DGAs over $R[U,V]$ or $R[U^{\pm 1},V^{\pm
  1}]$, we recall here the definition of stable tame isomorphism,
stated for DGAs over general rings.

\begin{definition}
Let $\mathcal{R}$ be a commutative ring with unit, and let
$(\alg,\d),(\tilde{\alg},\tilde{\d})$ be
DGAs (more precisely, finitely generated semifree noncommutative
unital differential graded
algebras) over $\mathcal{R}$: $\alg$ has a distinguished set of
generators $\{a_1,\ldots,a_n\}$ and is generated as an
$\mathcal{R}$-module by words in the $a_i$'s, and similarly
$\tilde{\alg}$ has a distinguished set of generators
$\{\tilde{a}_1,\ldots,\tilde{a}_{\tilde{n}}\}$.

\begin{itemize}
\item
An \textit{elementary automorphism} of $\alg$ is an
$\mathcal{R}$-algebra automorphism of the form
\begin{align*}
a_i &\mapsto \alpha \,a_i + v & &\text{for some $i$}\\
a_j &\mapsto a_j & &\text{for all $j\neq i$},
\end{align*}
where $v$ is an element of $\alg$ not involving $a_i$, and $\alpha$ is
an invertible element of $\mathcal{R}$.
\item
In the case where $n=\tilde{n}$, a \textit{tame isomorphism} $\phi$ between
$(\alg,\d)$ and $(\tilde{\alg},\tilde{\d})$ is a composition of some number of
elementary automorphisms
of $\alg$ and the isomorphism $\alg \to \tilde{\alg}$ that sends the
generators $\{a_i\}$ to the generators $\{\tilde{a}_i\}$ in some
order, such that $\phi\circ\d = \tilde{\d}\circ\phi$.
\item
A \textit{stabilization} of $(\alg,\d)$ is a DGA generated by
$a_1,\ldots,a_n$ and two new generators $e_1,e_2$, with
$|e_1|=|e_2|+1$ and differential induced by the differential on $\alg$
along with $\d(e_1) = e_2$, $\d(e_2)=0$.
\item
Two DGAs are \textit{stable tame isomorphic} if, after some number of
stabilizations of each, there is a tame isomorphism between the
resulting DGAs.
\end{itemize}
\end{definition}

A key feature of stable tame isomorphism is that it is a special case
of quasi-isomorphism.

\begin{proposition}[cf.\ \cite{Ch02,ENS}]
If $(\alg,\d)$ and $(\tilde{\alg},\tilde{\d})$ are stable tame isomorphic
DGAs over $\mathcal{R}$,
then we have an isomorphism of $\mathcal{R}$-modules
$H_*(\alg,\d) \cong H_*(\tilde{\alg},\tilde{\d})$.
\label{prop:quasi}
\end{proposition}

We conclude this section with a couple of results about the
consistency of our definitions for the transverse DGA and transverse homology.

\begin{proposition}
In the transverse DGA $(\CTminus_*(B),\dminus(B))$, we have $(\dminus)^2=0$.
\label{prop:d2}
\end{proposition}

\begin{proposition}
The homology in degree $0$ of the DGAs
$\CTminus_*(B)$, $\CThat_*(B)$, $\CTdoublehat_*(B)$, $\CTinfty_*(B)$ agrees
with $\HTminus_0(B)$, $\HThat_*(B)$, $\HTdoublehat_*(B)$, $\HTinfty_*(B)$ as
given in Definition~\ref{def:deg0}.
\label{prop:deg0}
\end{proposition}

Both of these results are easy consequences of the following algebraic
lemma, which is used throughout this paper.

\begin{lemma} (cf.\ \cite[Prop.~4.7]{NgKCH1}, \cite[Prop.~2.10]{NgKCH2})
Let $\phi_B(\A_\geq),\phi_B(\A_\leq)$ denote the images of
$\A_\geq,\A_\leq$ under $\phi$; in particular, these are lower and upper
triangular matrices, respectively, with $-1$ along the diagonal. Then
we have
\label{lem:philphir}
\begin{align*}
\phi_B(\A_\geq) &= \Phil_B \cdot \A_\geq \cdot \Phir_B \\
\phi_B(\A_\leq) &= \Phil_B \cdot \A_\leq \cdot \Phir_B.
\end{align*}
As a consequence, we also have
\begin{align*}
\phi_B(\Ahat) &= \Phil_B \cdot \Ahat \cdot \Phir_B \\
\phi_B(\Acheck) &= \Phil_B \cdot \Acheck \cdot \Phir_B.
\end{align*}
\end{lemma}

\begin{proof}
This lemma is implicit in the proof of Theorem~2.10 in \cite{Ngframed}, but
can be explicitly proven by induction by exactly following the proof
of Proposition~4.7 in \cite{NgKCH1}: first verify the identities explicitly
for $B = \sigma_k$, then use the chain rule for $\Phil,\Phir$ to
deduce the identities for $B = \sigma_k^{-1}$, and then for $B =
B_1B_2$ assuming the identities hold for $B=B_1$ and $B=B_2$.
\end{proof}

\begin{proof}[Proof of Proposition~\ref{prop:d2}]
It suffices to show that $(\dminus)^2(\E) = (\dminus)^2(\F) = 0$.
From the definitions of $\hat{\B},\check{\B},\dminus(\B)$, we see that
$\dminus(\hat{\B}) = \Ahat - \Lam_B \cdot \phi_B(\Ahat) \cdot
\Lam_B^{-1}$ and $\dminus(\check{\B}) = \Acheck - \Lam_B \cdot
\phi_B(\Acheck) \cdot \Lam_B^{-1}$. The desired result is then a
consequence of Lemma~\ref{lem:philphir}.
\end{proof}

\begin{proof}[Proof of Proposition~\ref{prop:deg0}]
(See Proposition~\ref{prop:modifiedDGA} for another proof.)
Since the transverse DGA is supported in nonnegative degree, it
suffices to check that the entries of the matrix
$\A - \Lam_B \cdot \phi_B(\A) \cdot \Lam_B^{-1}$ are in the ideal of
$\alg_n$ generated by the entries of
$\Ahat -
\Lam_B\cdot\Phil_B\cdot\Acheck$ and $\Acheck - \Ahat \cdot \Phir_B
\cdot \Lam_B^{-1}$. Now by Lemma~\ref{lem:philphir}, we have
\begin{align*}
\Ahat - \Lam_B \cdot \phi_B(\Ahat) \cdot \Lam_B^{-1} &=
(\Ahat -
\Lam_B\cdot\Phil_B\cdot\Acheck) + (\Lam_B \cdot \Phil_B) \cdot
(\Acheck - \Ahat \cdot \Phir_B \cdot \Lam_B^{-1}) \\
\Acheck - \Lam_B \cdot \phi_B(\Acheck) \cdot \Lam_B^{-1} &=
(\Acheck - \Ahat \cdot \Phir_B
\cdot \Lam_B^{-1}) + (\Ahat -
\Lam_B\cdot\Phil_B\cdot\Acheck) \cdot (\Phir_B \cdot \Lam_B^{-1}).
\end{align*}
But the matrix $\A - \Lam_B \cdot \phi_B(\A) \cdot \Lam_B^{-1}$ is
zero along the diagonal,
agrees below the diagonal with the first of these matrices, and agrees
above
the diagonal with $\mu$ times the second of these matrices.
\end{proof}

\section{Proofs of Invariance}
\label{sec:proofs}

In this section, we prove the main invariance theorems,
Theorems~\ref{thm:main} and~\ref{thm:infty}, along with some other
isomorphism results. We first establish some
auxiliary results giving alternate forms for the transverse DGA in
Section~\ref{ssec:modifiedDGA}. In Section~\ref{ssec:invproof1}, we
use these to prove the invariance of degree-$0$
transverse homology, which is technically a corollary of the main
invariance theorems but has a more streamlined proof,
and which is all we need to deduce the applications in
Section~\ref{sec:applications}. In Section~\ref{ssec:invproof2},
we present the full proofs of invariance, which are somewhat messy and
may probably be skipped with impunity by most readers. Finally, in
Section~\ref{ssec:eens}, we describe the version of the transverse DGA
presented in \cite{EENStransverse}, which is slightly different from
the transverse DGA in this paper, and prove that the two are stable
tame isomorphic.

\subsection{Equivalent forms for the transverse DGA}
\label{ssec:modifiedDGA}

In this subsection, we prove a few auxiliary results that show that
various DGAs related to the transverse DGA are stable tame
isomorphic. These results will be used in the invariance proofs in
Sections~\ref{ssec:invproof1}, \ref{ssec:invproof2}, and
\ref{ssec:eens}, and also in Section~\ref{sec:symmetries}.

First, it will sometimes be useful to replace the diagonal matrix $\Lam_B$ in
the definition of the transverse DGA by some other diagonal
matrix. Here we prove that the transverse DGA depends only on the
determinant of $\Lam_B$ rather than its particular entries.
Recall that we only consider the case of knots in this paper; the
corresponding result for multi-component links states that the
transverse DGA depends only on the determinants of the submatrices of
$\Lam_B$ corresponding to the various link components.

\begin{proposition}
Let $B$ be a braid.
\label{prop:lam}
Replace $\Lam_B$ in the definition of the
transverse DGA of $B$,
Definition~\ref{def:transDGA}, by an arbitrary diagonal matrix $\Lam$ with
invertible determinant. If the closure of $B$ is a single-component
knot, then the transverse DGA, up to tame isomorphism, depends
only on $\det\Lam$.
\end{proposition}

\begin{proof}
This is a straightforward extension of Proposition 3.1 in
\cite{Ngframed}, whose proof we follow here. For $B\in
B_n$, let $s(B)\in S_n$ denote
the permutation of $\{1,\ldots,n\}$ corresponding to $B$ under the
usual map $B_n \to S_n$. If $v$ is a vector of length $n$, let
$\Del(v)$ denote the diagonal $n\times n$ matrix whose diagonal
entries are $v$, and write
$s(B)v$ for the vector that results from permuting the entries of $v$
by $s(B)$. We recall \cite[Lemma~3.2]{Ngframed}: if we define $\tilde{\A}
= \Del(v) \cdot \A \cdot \Del(v)^{-1}$, then
\begin{align*}
\Phil_B(\tilde{\A}) &= \Del(s(B)v) \cdot \Phil_B(\A) \cdot
\Del(v)^{-1} \\
\Phir_B(\tilde{\A}) &= \Del(v) \cdot \Phir_B(\A) \cdot
\Del(s(B)v)^{-1}.
\end{align*}
Here, as in \cite{NgKCH1,Ngframed}, $\Phil_B(\tilde{\A}),\Phir_B(\tilde{\A})$ denote the result of replacing the generators $a_{ij}$ by $\tilde{\A}_{ij}$ (for all $i,j$) in the matrices $\Phil_B,\Phir_B$.

Now given any vector $v$ of length $n$ with invertible entries, let
$(\CTminus,\d^-)$ and $(\widetilde{\mathit{CT}}^{ -},\tilde{\d}^-)$ be the transverse DGAs
for $B$ with $\Lam_B$ replaced by $\Lam$ and $\tilde{\Lam} = \Del(v)
\cdot\Del(s(B)v)^{-1}\cdot \Lam$,
respectively, for some diagonal matrix $\Lam$ with invertible diagonal
entries. Then the identification of the algebras $\CTminus$ with
$\widetilde{\mathit{CT}}^{ -}$
given in matrix form by $\tilde{\A} = \Del(v) \A \Del(v)^{-1}$,
$\tilde{\B} = \Del(v) \B \Del(v)^{-1}$,
$\tilde{\C} = \Del(v) \C \Del(v)^{-1}$,
$\tilde{\D} = \Del(v) \D \Del(v)^{-1}$,
$\tilde{\E} = \Del(v) \E \Del(v)^{-1}$,
$\tilde{\F} = \Del(v) \F \Del(v)^{-1}$ provides a tame isomorphism
between $(\CTminus,\d^-)$ and $(\widetilde{\mathit{CT}}^{ -},\tilde{\d}^-)$:
\begin{align*}
\tilde{\d}^-(\tilde{\C}) &=
\hat{\tilde{\A}} - \tilde{\Lam} \cdot \Phil_B(\tilde{\A}) \cdot
\check{\tilde{\A}} \\
&= \Del(v) \cdot \hat{\A} \cdot \Del(v)^{-1} -
\tilde{\Lam} \cdot \Del(s(B)v) \cdot \Phil_B(\A) \cdot \check{\A} \cdot
\Del(v)^{-1} \\
&= \Del(v) \cdot (\hat{\A} - \Lam \cdot \Phil_B(\A) \cdot
\check{\A}) \cdot \Del(v)^{-1} \\
&= \d^-(\tilde{\C}),
\end{align*}
with similar computations for the differentials of $\B,\D,\E,\F$.

Thus, in the definition of the transverse DGA, we can replace
$\Lam_B$ by any matrix of the form $\Del(v) \cdot \Del(s(B)v)^{-1}
\cdot \Lam_B$, up to tame isomorphism. Since the closure of $B$ is a
knot, $s(B)$ is an $n$-cycle; it follows that for any diagonal
$\Lam$ with the same determinant as $\Lam_B$, one can choose a
vector $v$ with $\Lam = \Del(v) \cdot \Del(s(B)v)^{-1}
\cdot \Lam_B$. (We need $\Lam_{ii} = v_i v_{s(B)(i)}^{-1}
(\Lam_B)_{ii}$ for $i=1,\ldots,n$, and this formula allows us to define
$v$ up to an overall multiplicative factor.) The proposition follows.
\end{proof}

\begin{corollary}
Let $B\in B_n$ be a braid whose closure is a knot, and let $\Lam$ be a
diagonal $n\times n$ matrix with invertible determinant. Up to
isomorphism, the
degree-$0$ transverse homology
\[
\HTminus_0(B) = \alg_n \, / \, (\Ahat -
\Lam\cdot\Phil_B\cdot\Acheck,
\Acheck - \Ahat \cdot \Phir_B \cdot \Lam^{-1})
\]
depends only on $\det\Lam$, with corresponding results for the other
flavors of degree-$0$ transverse homology.
\label{cor:lam}
\end{corollary}

Next, there are several variants of the transverse DGA, akin to the
``modified framed DGA'' from \cite{Ngframed}, that are stable tame
isomorphic to the transverse DGA but have fewer generators. It will
occasionally be convenient to use one of the variants instead of the
original transverse DGA.

\begin{proposition}
Let $B$ be a braid.
\label{prop:modifiedDGA}
The following DGAs over $R[U,V]$ are stable tame isomorphic:
\begin{enumerate}
\item \label{DGA1}
the transverse DGA from Definition~\ref{def:transDGA};
\item
the DGA with generators and differential: \label{DGA3}
\begin{itemize}
\item
$\{a_{ij}\}_{1\leq i\neq j\leq n}$ of degree $0$ with $\d^-(\A) = 0$,
\item
$\{c_{ij}\}_{1\leq i,j\leq n}$ of degree $1$ with
$\d^-(\C) = \Ahat-\Lam_B \cdot \Phil_B \cdot \Acheck$,
\item
$\{d_{ij}\}_{1\leq i,j\leq n}$ of degree $1$ with
$\d^-(\D) = \Acheck-\Ahat\cdot\Phir_B\cdot\Lam_B^{-1}$,
\item
$\{e_{ij}\}_{1\leq i\leq j\leq n}$ of degree $2$ with
$\d^-(e_{ii}) = (\C+\Lam_B\cdot\Phil_B\cdot\D)_{ii}$ and
\[
\d^-(e_{ij}) = (\C-U \D+\Lam_B\cdot\Phil_B\cdot\D-
U \C\cdot\Phir_B\cdot\Lam_B^{-1})_{ij}
\]
for $i<j$,
\item
$\{f_{ij}\}_{1\leq j\leq i\leq n}$ of degree $2$ with
$\d^-(f_{ii}) = (\D+\C\cdot\Phir_B\cdot\Lam_B^{-1})_{ii}$ and
\[
\d^-(f_{ij}) = (\D-V \C+\C\cdot\Phir_B\cdot\Lam_B^{-1} - V \Lam_B\cdot\Phil_B\cdot\D)_{ij}
\]
for $j<i$;
\end{itemize}
\item
the DGA with generators and differential: \label{DGA2}
\begin{itemize}
\item
$\{a_{ij}\}_{1\leq i\neq j\leq n}$ of degree $0$ with $\d^-(\A) = 0$,
\item
$\{b_{ij}\}_{1\leq i\neq j\leq n}$ of degree $1$ with $\d^-(\B) =
\Lam_B^{-1} \cdot \A \cdot \Lam_B - \phi_B(\A)$,
\item
$\{d_{ij}\}_{1\leq i,j\leq n}$ of degree $1$ with
$\d^-(\D) = \Acheck \cdot \Lam_B -\Ahat \cdot \Phir_B$,
\item
$\{f_{ij}\}_{1\leq i,j\leq n}$ of degree $2$ with
\[
\d^-(\F) =
\check{\B}\cdot(\Phir_B)^{-1} - \hat{\B} \cdot \Lam_B^{-1} +
\Phil_B \cdot \D \cdot \Lam_B^{-1} - \Lam_B^{-1} \cdot \D \cdot
(\Phir_B)^{-1}.
\]
\end{itemize}
\end{enumerate}
\end{proposition}

\begin{proof}
To obtain DGA (\ref{DGA3}) from the transverse DGA (\ref{DGA1}), apply the tame automorphisms
\[
b_{ij} \mapsto b_{ij} + (\C+\Lam_B\cdot\Phil_B\cdot\D)_{ij}
\]
for $i>j$,
\[
b_{ij} \mapsto b_{ij} + \mu^{-1} (\D+\C\cdot\Phir_B\cdot\Lam_B^{-1})_{ij}
\]
for $i<j$, to the transverse DGA. The result has $\d^-(e_{ij}) = b_{ij}$ for $i>j$, $\d^-(f_{ij}) = b_{ij}$ for $i<j$, and $\d^-(\B) = 0$ (since $(\d^-)^2=0$). We can then destabilize by removing $b_{ij},e_{ij}$ for $i>j$ and $b_{ij},f_{ij}$ for $i<j$, and the result is DGA (\ref{DGA3}).

To obtain DGA (\ref{DGA2}) from the transverse DGA (\ref{DGA1}), successively apply the tame automorphisms
\begin{align*}
\C &\mapsto \C + \hat{\B} - \Lam_B \cdot \Phil_B \cdot \D \\
\D & \mapsto \D \cdot \Lam_B^{-1} \\
\B & \mapsto \Lam_B \cdot \B \cdot \Lam_B^{-1} \\
\F &\mapsto (\Lam_B\cdot\F+\E)\cdot\Phir_B\cdot\Lam_B^{-1}
\end{align*}
to the transverse DGA. The result has
$\d^-(\E) = -\C$, $\d^-(\C) = 0$ (since $(\d^-)^2=0$), and $\d^-(\B),
\d^-(\D), \d^-(\F)$ as given in the statement of the proposition.
We can then destabilize by removing the $c$ and $e$ generators, and
the result is DGA (\ref{DGA2}).
\end{proof}

DGA (\ref{DGA3}) from Proposition~\ref{prop:modifiedDGA} is the
analogue of the framed DGA from
\cite{Ngframed} and is convenient to use in the invariance proofs
under braid stabilization presented in Section~\ref{ssec:invproof2}
below. DGA (\ref{DGA2}) is, up to signs and powers of $\lambda$ and
$\mu$, the version of the transverse DGA considered in
\cite{EENStransverse}; see Section~\ref{ssec:eens} below for the exact
relation.

\subsection{Invariance proofs I: degree-$0$ transverse homology}
\label{ssec:invproof1}

Before embarking on the full invariance proofs for the transverse DGA,
we first prove the invariance of degree-$0$ transverse homology
$\mathit{HT}_0$. This is technically superfluous, since it follows
from the invariance of the DGA, but we have included it because it is
a simplified (and more readable) version of the full invariance
proof. In addition, as mentioned before, the applications of
transverse homology that we present in Section~\ref{sec:applications}
rely only on the invariance of $\mathit{HT}_0$, and the reader
interested in applications rather than the theory behind transverse
homology can skip the full invariance proofs in favor of the proofs in
this section.

\begin{proposition}
The degree-$0$ transverse homology $\HTminus_0(B)$ is an invariant of
transverse knots. More precisely, if braids $B,\tilde{B}$ are related by
some sequence of conjugation and positive de/stabilization, then there
is an isomorphism of $R[U,V]$-algebras
\label{prop:HT0inv}
\[
\HTminus_0(B) \cong \HTminus_0(\tilde{B}).
\]
\end{proposition}

\begin{proof}
Let $B\in B_n$. It suffices to consider $B,\tilde{B}$ related by one of the following:
$\tilde{B} = \sigma_k^{-1} B \sigma_k$ for some $k$; $\tilde{B} = B
\sigma_0$, where we view $B\subset B_{n+1}$ with the strands in
$B_{n+1}$ numbered $0,\ldots,n$ instead of $1,\ldots,n+1$.

\vspace{11pt}

\noindent \textsc{Case 1: $\tilde{B} = \sigma_k^{-1} B \sigma_k$.}

We follow the proofs of Theorem~4.10 from \cite{NgKCH1} and
Theorem~2.7 from \cite{Ngframed}. Define $\Lam$ to be $\Lam_B = \diag(\lambda\mu^{-w(B)},1,\ldots,1)$ if $k\neq 1$, and
$\diag(1,\lambda\mu^{-w(B)},1,\ldots,1)$ if $k=1$. Write $\I,\tilde{\I}$ for the ideals
\begin{align*}
\I &= (\Ahat -
\Lam_B\cdot\Phil_B\cdot\Acheck,
\Acheck - \Ahat \cdot \Phir_B \cdot \Lam_B^{-1})
\\
\tilde{\I} &= (\Ahat - \Lam\cdot\Phil_{\tilde{B}}\cdot\Acheck,
\Acheck-\Ahat\cdot\Phir_{\tilde{B}}\cdot\Lam^{-1})
\end{align*}
in $\alg_n$.
By Corollary~\ref{cor:lam}, we have $\HTminus_0(B) \cong \alg_n/\I$ and $\HTminus_0(\tilde{B}) \cong \alg_n/\tilde{\I}$.

We claim that $\phi_{\sigma_k} :\thinspace \alg_n \to \alg_n$ maps
$\tilde{\I}$ into $\I$; a similar argument shows that $\phi_{\sigma_k}^{-1}$ maps
$\I$ into $\tilde{\I}$, which proves that $\phi_{\sigma_k}$ induces an
isomorphism $\HTminus_0(B) \to \HTminus_0(\tilde{B})$.
To prove the claim, we will show that the entries of
\[
\mathbf{M} = \phi_{\sigma_k}(\Ahat) - \Lam \cdot
\Phil_{\tilde{B}}(\phi_{\sigma_k}(\A)) \cdot \phi_{\sigma_k}(\Acheck)
\]
are in $\I$. Similarly, the entries of $\phi_{\sigma_k}(\Acheck) -
\phi_{\sigma_k}(\Ahat) \cdot \Phir_{\tilde{B}}(\phi_{\sigma_k}(\A))
\cdot \Lam^{-1}$ are in $\I$, and the claim follows.

We have
\begin{align*}
\mathbf{M} &= \Phil_{\sigma_k}(\A) \Ahat \Phir_{\sigma_k}(\A) -
\Lam \left( \Phil_{\sigma_k}(\phi_B(\A)) \Phil_B(\A)
  (\Phil_{\sigma_k}(\A))^{-1} \right) (\Phil_{\sigma_k}(\A) \Acheck
\Phir_{\sigma_k}(\A)) \\
&= \Phil_{\sigma_k}(\A)
\left\{\Ahat-\Lam_B\Phil_B(\A)\Acheck\right\}\Phir_{\sigma_k}(\A) \\
& \qquad - \left\{\Lam
  \Phil_{\sigma_k}(\phi_B(\A))-\Phil_{\sigma_k}(\A) \Lam_B \right\}
\Phil_B(\A) \Acheck \Phir_{\sigma_k}(\A).
\end{align*}
But both terms in braces in this last expression have entries in
$\I$. This is clear for the first term in braces;
by the expression for $\Phil_{\sigma_k}$ from \cite[Lemma~4.6]{NgKCH1},
the second term in braces is the $0$ matrix except in the $kk$ entry, where it is
$a_{k+1,k}-\phi_B(a_{k+1,k})$ (or $\lambda\mu^{-w(B)}a_{21}-\phi_B(a_{21})$ if $k=1$),
which is a scalar multiple of the $k+1,k$ entry of the matrix
$\A-\Lam_B\cdot\phi_B(\A)\cdot\Lam_B^{-1}$ and hence
in $\I$ by the proof of Proposition~\ref{prop:deg0}.

\vspace{11pt}

\noindent \textsc{Case 2: $\tilde{B} = B \sigma_0$.}
Write $\tilde{\alg}$ for the algebra over $R[U,V]$ obtained from
$\alg_n$ by adding generators $a_{i0},a_{0i}$ for $1\leq i\leq
n$. Also write $\dminus$, $\tilde{\partial}^-$ for the differentials
on the transverse DGAs of $B$ and $\tilde{B}$, respectively. Then we
can write $\HTminus_0(B) = \alg_n/\I$ and
$\HTminus_0(\tilde{B}) = \tilde{\alg}/\tilde{\I}$, where $\I$ is
generated by $\dminus(c_{ij}),\dminus(d_{ij})$ for $1\leq i\neq j\leq
n$, and $\tilde{\I}$ is generated by
$\tilde{\partial}^-(c_{ij}),\tilde{\partial}^-(d_{ij})$ for $0\leq
i\neq j\leq n$. Using the expressions for
$\Phil_{B\sigma_0},\Phir_{B\sigma_0}$ from the proof of Theorem~4.10
in \cite{NgKCH1}, we calculate, for $1\leq i\leq n$ and $2\leq j\leq n$:
\begin{align*}
\tilde{\partial}^-(c_{00}) &= -1-\mu U-\lambda \mu^{-w(B)-1}
(\phi(a_{10})(V+\mu)-V\Phil_{1\ell}a_{\ell 0}) \\
& \quad = -1-\mu U-\lambda \mu^{-w(B)}
\Phil_{1\ell}a_{\ell 0}  \\
\tilde{\partial}^-(c_{0i}) &=
\mu U a_{0i} - \lambda\mu^{-w(B)-1}
(-\mu\phi(a_{10})a_{0i}-\Phil_{1\ell} \check{a}_{\ell i})
 \\
\tilde{\partial}^-(c_{10}) &= a_{10}+V+\mu \\
\tilde{\partial}^-(c_{1i}) &= \hat{a}_{1i}-\mu a_{0i} \\
\tilde{\partial}^-(c_{j0}) &= a_{j0}-V\Phil_{j\ell}a_{\ell 0} \\
\tilde{\partial}^-(c_{ji}) &= \hat{a}_{ji}-\Phil_{j\ell}\check{a}_{\ell i} \\
\tilde{\partial}^-(d_{00}) &=-V-\mu-\lambda^{-1}\mu^{w(B)+1}
((1+\mu U)\phi(a_{01})-\mu U a_{0\ell} \Phir_{\ell 1}) \\
& \quad = -V-\mu-\lambda^{-1}\mu^{w(B)+1}
a_{0\ell} \Phir_{\ell 1} \\
\tilde{\partial}^-(d_{i0}) &=
Va_{i0}-\lambda^{-1}\mu^{w(B)+1}
(-a_{i0}\phi(a_{01})
-\hat{a}_{i\ell}\Phir_{\ell 1})\\
\tilde{\partial}^-(d_{01}) &= \mu a_{01}+1+\mu U \\
\tilde{\partial}^-(d_{i1}) &= \check{a}_{i1}-a_{i0} \\
\tilde{\partial}^-(d_{0j}) &= \mu a_{0j}-\mu U a_{0\ell} \Phir_{\ell j} \\
\tilde{\partial}^-(d_{ij}) &=
\check{a}_{ij}-\hat{a}_{i\ell}\Phir_{\ell j}.
\end{align*}
Here $\hat{a}_{ij},\check{a}_{ij}$ are the $ij$ entries of
$\Ahat,\Acheck$ respectively, $\phi,\Phil,\Phir$ are shorthand for
$\phi_B,\Phil_B,\Phir_B$, and all appearances of $\ell$ are
understood to represent sums over $\ell=1,\ldots,n$.

From $\tilde{\partial}^-(c_{1i}),\tilde{\partial}^-(d_{i1})$, we see
that in $\tilde{\alg}/\tilde{\I}$, we have
$a_{0i} = \hat{a}_{1i}/\mu$ and $a_{i0} = \check{a}_{i1}$ for $1\leq
i\leq n$. Using these relations to replace all appearances of
$a_{0i},a_{i0}$, we find that the remaining relations in
$\tilde{\alg}/\tilde{\I}$ become precisely the generators of $\I$.
For instance,
\begin{align*}
\tilde{\partial}^-(c_{00}) &= -1 - \mu U-\lambda\mu^{-w(B)}
\Phil_{1\ell}\check{a}_{\ell 1} = \partial^-(c_{11}) \\
\tilde{\partial}^-(c_{0i}) &=
U\hat{a}_{1i}+\lambda\mu^{-w(B)-1}\left(\Phil_{1\ell} \check{a}_{\ell 1}
\hat{a}_{1i} + \Phil_{1\ell}\check{a}_{\ell i}\right)
= -\mu^{-1}\left((\d^-(c_{11}))\hat{a}_{1i}+\d^-(c_{1i})\right)
\end{align*}
and so forth. It
follows that $\tilde{\alg}/\tilde{\I} \cong \alg_n/\I$, as desired.
\end{proof}

We next establish the invariance result for degree-$0$ infinity transverse homology.

\begin{proposition}
The degree-$0$ infinity transverse homology $\HTinfty_0(B)$ is an invariant of
topological knots. More precisely, if braids $B,\tilde{B}$ have the
same knot closure, then there
is an isomorphism of $R[U^{\pm 1},V^{\pm 1}]$-algebras
\label{prop:HTinftyinv}
\[
\HTinfty_0(B) \cong \HTinfty_0(\tilde{B}).
\]
\end{proposition}

\begin{proof}
Because of Proposition~\ref{prop:HT0inv}, it suffices to establish the
result when $\tilde{B}$ is a negative stabilization of $B$: $\tilde{B}
= B\sigma_0^{-1}$. We use the same notation as the proof of
Proposition~\ref{prop:HT0inv} in the case $\tilde{B} =
B\sigma_0$. Using the expressions for
$\Phil_{B\sigma_0^{-1}},\Phir_{B\sigma_0^{-1}}$
from the proof of Theorem~4.10
in \cite{NgKCH1}, we calculate, for $1\leq i\leq n$ and $2\leq j\leq n$:
\begin{align*}
\tilde{\partial}^{\infty}(c_{00}) &=
-1-\mu U-\lambda \mu^{-w(B)+1} V(U/V)^{-(\sl(B)-1)/2}\Phil_{1\ell} a_{\ell 0}  \\
\tilde{\partial}^{\infty}(c_{0i}) &= \mu U
a_{0i}-\lambda\mu^{-w(B)+1}(U/V)^{-(\sl(B)-1)/2}\Phil_{1\ell} \check{a}_{\ell i}  \\
\tilde{\partial}^{\infty}(c_{10}) &=
a_{10}-(V+\mu-V\phi(a_{01})\Phil_{1\ell}a_{\ell 0})  \\
\tilde{\partial}^{\infty}(c_{1i}) &=
\hat{a}_{1i}-(-\mu a_{0i}-\phi(a_{01})\Phil_{1\ell}\check{a}_{\ell i})  \\
\tilde{\partial}^{\infty}(c_{j0}) &= a_{j0}-V\Phil_{j\ell}a_{\ell 0}  \\
\tilde{\partial}^{\infty}(c_{ji}) &=
\hat{a}_{ji}-\Phil_{j\ell}\check{a}_{\ell i}  \\
\tilde{\partial}^{\infty}(d_{00}) &= -V-\mu-\lambda^{-1}\mu^{w(B)}U(U/V)^{(\sl(B)-1)/2}
a_{0\ell}\Phir_{\ell 1}  \\
\tilde{\partial}^{\infty}(d_{i0}) &= Va_{i0}-\lambda^{-1}\mu^{w(B)-1}
(U/V)^{(\sl(B)-1)/2}\hat{a}_{i\ell}\Phir_{\ell 1}  \\
\tilde{\partial}^{\infty}(d_{01}) &= \mu a_{01} - (1+\mu U-\mu U
a_{0\ell} \Phir_{\ell 1}\phi(a_{10}))  \\
\tilde{\partial}^{\infty}(d_{i1}) &=
\check{a}_{i1}-(-a_{i0}-\hat{a}_{i\ell}\Phir_{\ell 1}\phi(a_{10}))  \\
\tilde{\partial}^{\infty}(d_{0j}) &= \mu a_{0j}-\mu U
a_{0\ell}\Phir_{\ell j}  \\
\tilde{\partial}^{\infty}(d_{ij}) &= \check{a}_{ij}-\hat{a}_{i\ell}
\Phir_{\ell j}.
\end{align*}
From $\tilde{\partial}^{\infty}(c_{0i}),\tilde{\partial}^{\infty}(d_{i0})$, we see
that in $\tilde{\alg}/\tilde{\I}$, we have
\begin{align*}
a_{0i} &= \lambda\mu^{-w(B)}V^{-1}(U/V)^{-(\sl(B)+1)/2}
\Phil_{1\ell}\check{a}_{\ell i} \\
a_{i0} &= \lambda^{-1}\mu^{w(B)-1}U^{-1}(U/V)^{(\sl(B)+1)/2}
\hat{a}_{i\ell}\Phir_{\ell 1}
\end{align*} for $1\leq
i\leq n$. Using these relations to replace all appearances of
$a_{0i},a_{i0}$, we find that the relations
$\tilde{\partial}^\infty(c_{00}) = 0$ and
$\tilde{\partial}^\infty(d_{00}) = 0$ become $0=0$, and the
remaining relations in
$\tilde{\alg}/\tilde{\I}$ become precisely the generators of $\I$. It
follows that $\tilde{\alg}/\tilde{\I} \cong \alg_n/\I$, as desired.
\end{proof}

\subsection{Invariance proofs II: the full transverse DGA}
\label{ssec:invproof2}

Here we extend the arguments from Section~\ref{ssec:invproof1} to
prove invariance for the full transverse DGA, Theorems~\ref{thm:main}
and~\ref{thm:infty}. We provide detailed outlines for the invariance proofs, which follow the proof of Theorem 2.7 in \cite{Ngframed}, and leave some amount of easy but tedious checking to the reader. We remark that the proofs given here reduce, upon setting $U=V=1$, to a proof of invariance of the knot DGA that is very similar to, but slightly different from, the original proof from \cite[Thm.~2.7]{Ngframed}. The few changes are changes of convenience for our current setup.

\begin{proof}[Outline of proof of Theorem~\ref{thm:main}]

As in the proof of Proposition~\ref{prop:HT0inv}, it suffices to prove
invariance for the transverse DGA under stable tame isomorphism for
braids $B,\tilde{B}$ related by either conjugation or positive stabilization.

\vspace{11pt}

\noindent \textsc{Case 1: $\tilde{B} = \sigma_k^{-1} B \sigma_k$.}

Let $(\alg=\CTminus,\d^-)$ denote the transverse DGA of $B$, and let $(\tilde{\alg}=\CTtildeminus,\tilde{\d}^-)$ denote the transverse DGA of $\tilde{B}$ but with $\Lam_{\tilde{B}}$ replaced by $\tilde{\Lam}$, defined to be $\diag(\lambda\mu^{-w(B)},1,\ldots,1)$ if $k\neq 1$ and $\diag(1,\lambda\mu^{-w(B)},1,\ldots,1)$ if $k=1$; this replacement is allowed by Proposition~\ref{prop:lam}. We claim that $(\alg,\d^-)$ and $(\tilde{\alg},\tilde{\d}^-)$ are tamely isomorphic.

The tame isomorphism, which is nearly identical to the corresponding map considered in the proof of \cite[Thm.~2.7]{Ngframed}\footnote{Note that our $e$ and $f$ generators are respectively the $f$ and $e$ generators in \cite{Ngframed}.}, is given by the identification:
\begin{align*}
\tilde{\A} &= \phi_{\sigma_k}(\A) \\
\tilde{\B} &= \Phil_{\sigma_k}(\epsilon_k^{-1}\phi_B(\A)) \cdot \B \cdot \Phir_{\sigma_k}(\epsilon_k\phi_B(\A)) + \Phil_{\sigma_k}(\A)\cdot\A\cdot\Th^R_k \\
& \qquad + \Th^L_k\cdot \A\cdot\Phir_{\sigma_k}(\epsilon_k\phi_B(\A))
+ \d^-(\Th^L_k \cdot \Th^R_k) \\
\tilde{\C} &= \Phil_{\sigma_k}(\epsilon_k^{-1}\phi_B(\A)) \cdot \C \cdot \Phir_{\sigma_k}(\A) + \Th^L_k\cdot\Ahat\cdot\Phir_{\sigma_k}(\A) \\
\tilde{\D} &= \Phil_{\sigma_k}(\A) \cdot \D \cdot \Phir_{\sigma_k}(\epsilon_k\phi_B(\A)) + \Phil_{\sigma_k}(\A) \cdot \Acheck \cdot \Th^R_k \\
\tilde{\E} &= \Phil_{\sigma_k}(\epsilon_k^{-1}\phi_B(\A)) \cdot \E \cdot \Phir_{\sigma_k}(\epsilon_k\phi_B(\A)) + \Phil_{\sigma_k}(\epsilon_k^{-1}\phi_B(\A)) \cdot C \cdot \Th^R_k + \Th^L_k\cdot(\Ahat+\mu U)\cdot\Th^R_k \\
\tilde{\F} &= \Phil_{\sigma_k}(\epsilon_k^{-1}\phi_B(\A)) \cdot \F \cdot \Phir_{\sigma_k}(\epsilon_k\phi_B(\A)) - \Th^L_k \cdot \D \cdot \Phir_{\sigma_k}(\epsilon_k\phi_B(\A)) + \mu \Th^L_k\cdot\Th^R_k.
\end{align*}
Here, as in \cite{Ngframed}, $\epsilon_k$ is $\lambda \mu^{-w(B)}$ if $k=1$ and $1$ otherwise;
$\Phil_{\sigma_k}(\epsilon_k^{-1}\phi_B(\A)),\Phir_{\sigma_k}(\epsilon_k\phi_B(\A))$ are the matrices $\Phil_{\sigma_k}(\A),\Phir_{\sigma_k}(\A)$ with $a_{k+1,k},a_{k,k+1}$ replaced by $\epsilon_k^{-1}\phi_B(a_{k+1,k}),\epsilon_k\phi_B(a_{k,k+1})$; and $\Th^L_k,\Th^R_k$ are the matrices that are identically zero except in the $(k,k)$ entry, where they are $-b_{k+1,k},-b_{k,k+1}$. (Note however that $\A$ here, which has $-2$ entries along the diagonal, is the result of setting $\mu=-1$ in the matrix $A$ from \cite{Ngframed}.)

We leave the verification that this identification intertwines $\d^-$ and $\tilde{\d}^-$ to the reader. In addition to the identities provided in the proof of \cite[Thm.~2.7]{Ngframed}, the following two identities are useful in this regard:
\begin{align*}
\hat{\tilde{\B}} &= \Phil_{\sigma_k}(\epsilon_k^{-1}\phi_B(\A)) \cdot \hat{\B} \cdot \Phir_{\sigma_k}(\epsilon_k\phi_B(\A)) + \Phil_{\sigma_k}(\A)\cdot\Ahat\cdot\Th^R_k \\
& \qquad + \Th^L_k\cdot \Ahat\cdot\Phir_{\sigma_k}(\epsilon_k\phi_B(\A))
+ \mu U \,\d^-(\Th^L_k \cdot \Th^R_k) \\
\check{\tilde{\B}} &= \Phil_{\sigma_k}(\epsilon_k^{-1}\phi_B(\A)) \cdot \check{\B} \cdot \Phir_{\sigma_k}(\epsilon_k\phi_B(\A)) + \Phil_{\sigma_k}(\A)\cdot\Acheck\cdot\Th^R_k \\
& \qquad + \Th^L_k\cdot \Acheck\cdot\Phir_{\sigma_k}(\epsilon_k\phi_B(\A))
+ \mu \,\d^-(\Th^L_k \cdot \Th^R_k).
\end{align*}

\vspace{11pt}

\noindent \textsc{Case 2: $\tilde{B} = B \sigma_0$.}

Let $(\alg=\CTminus,\d^-)$, $(\tilde{\alg}=\CTtildeminus,\tilde{\d}^-)$ denote the modified versions of
the transverse
DGAs of $B$, $\tilde{B}$, respectively, given by (\ref{DGA3}) in
Proposition~\ref{prop:modifiedDGA}. We wish to show that
$(\alg,\d^-)$ and $(\tilde{\alg},\tilde{\d}^-)$ are stable tame
isomorphic.

The generators of $\alg$ are: $a_{ij}$ for $1\leq i\neq j\leq n$; $c_{ij},d_{ij}$ for $1\leq i,j\leq n$; and $e_{ij},f_{ji}$ for $1\leq i\leq j\leq n$. In order to give $\alg$ as many generators as $\tilde{\alg}$ has, add to $\alg$ the generators
\[
a_{0i},a_{i0},c_{00},c_{0i},c_{i0},d_{00},d_{0i},d_{i0},e_{00},e_{0i},f_{00},f_{i0}
\]
for $1\leq i\leq n$, and extend the differential on $\alg$ to these generators by: $\d^- c_{0i} = -\mu a_{0i}+\hat{a}_{1i}$, $\d^- d_{i0} = -a_{i0}+\check{a}_{i1}$ for $1\leq i\leq n$; $\d^- e_{00} = c_{00}$, $\d^- f_{00} = d_{00}$, $\d^- e_{01} = c_{10}$, $\d^- f_{10} = d_{01}$; $\d^- e_{0j} = d_{0j}$, $\d^- f_{j0} = c_{j0}$ for $2\leq j\leq n$; and the differential on the other new generators is $0$. The resulting DGA, which we also write as $(\alg,\d^-)$, is stable tame isomorphic to the original DGA for $B$.

We will present a tame isomorphism between the new $(\alg,\d^-)$ and $(\tilde{\alg},\tilde{\d}^-)$.
For clarity, add tildes to the generators of $\tilde{\alg}$: $\tilde{a}_{ij}$ for $0\leq i\neq j\leq n$; $\tilde{c}_{ij},\tilde{d}_{ij}$ for $0\leq i,j\leq n$; $\tilde{e}_{ij},\tilde{f}_{ji}$ for $0\leq i\leq j\leq n$. The tame isomorphism is the identification between $\alg$ and $\tilde{\alg}$ given as follows: $a_{i_1i_2} = \tilde{a}_{i_1i_2}$ for $0\leq i_1\neq i_2\leq n$ and
\begin{alignat*}{2}
c_{00} &= \tilde{c}_{00} - \lambda\mu^{-w(B)}(\Phil_B)_{1\ell} \tilde{d}_{\ell 1}-c_{11} \qquad &
d_{00} &= \tilde{d}_{00} - \lambda^{-1}\mu^{w(B)}\tilde{c}_{1\ell}(\Phir_B)_{\ell 1}-d_{11} \\
c_{0i} &= \tilde{c}_{1i} &
d_{i0} &= \tilde{d}_{i1} \\
c_{10} &= \tilde{d}_{01} + \tilde{c}_{11} &
d_{01} &= \tilde{c}_{10} + \tilde{d}_{11} \\
c_{1i} &= -\mu \tilde{c}_{0i} + \tilde{c}_{1i} - \mu \tilde{c}_{00} a_{0i} &
d_{i1} &= -\mu^{-1} \tilde{d}_{i0}+\tilde{d}_{i1}-\mu^{-1}a_{i0}\tilde{d}_{00} \\
c_{j0} &=
\tilde{c}_{j0}+\tilde{d}_{j1}-V\tilde{c}_{j1}-V(\Phil_B)_{j\ell}\tilde{d}_{\ell
  1} &
d_{0j} &=
\tilde{d}_{0j}+\tilde{c}_{1j}-U\tilde{d}_{1j}-U\tilde{c}_{1\ell}(\Phir_B)_{\ell
  j}\\
c_{ji} &= \tilde{c}_{ji} &
d_{ij} &= \tilde{d}_{ij} \\
e_{00} &= \mu\tilde{e}_{01}-\tilde{e}_{11}-\tilde{c}_{00}\tilde{d}_{01} &
f_{00} &= \mu^{-1}\tilde{f}_{10}-\tilde{f}_{11}+\mu^{-1}\tilde{c}_{10}\tilde{d}_{00} \\
e_{0i} &= \tilde{e}_{1i} & f_{i0} &= \tilde{f}_{i1} \\
e_{11} &= \tilde{e}_{00}+\tilde{e}_{11}-\mu \tilde{e}_{01}+\tilde{c}_{00}\tilde{d}_{01} &
f_{11} &= \tilde{f}_{00} + \tilde{f}_{11} - \mu^{-1} \tilde{f}_{10} - \mu^{-1} \tilde{c}_{10}\tilde{d}_{00} \\
e_{1j} &= -\mu \tilde{e}_{0j} - \tilde{e}_{1j} + \tilde{c}_{00}\tilde{d}_{0j} &
f_{j1} &= -\mu^{-1}\tilde{f}_{j0} + \tilde{f}_{j1} - \mu^{-1} \tilde{c}_{j0} \tilde{d}_{00} \\
e_{j_1j_2} &= \tilde{e}_{j_1j_2} &
f_{j_2j_1} &= \tilde{f}_{j_2j_1}
\end{alignat*}
for $1\leq i\leq n$, $2\leq j\leq n$, and $2\leq j_1\leq j_2\leq
n$. Here expressions involving $\ell$ are summations over
$1\leq\ell\leq n$. 

The above identification intertwines the differentials $\d^-$ and
$\tilde{\d}^-$, a fact that we leave to the reader to verify. To complete the proof of the
theorem, we need to check that the identification represents a tame
isomorphism between $\alg$ and $\tilde{\alg}$. To this end, we can
express the identification as a composition of maps beginning at
$\tilde{\alg}$ and ending at $\alg$, where each map replaces some of the
generators of $\tilde{\alg}$ by generators of $\alg$. Beginning with
$\tilde{\alg}$, we successively replace generators as follows:
\begin{itemize}
\item
the $\tilde{a}$'s by the $a$'s;
\item
then $\tilde{e}_{00}$, $\tilde{f}_{00}$, $\tilde{e}_{0j}$, $\tilde{f}_{j0}$ by $e_{11}$, $f_{11}$, $e_{1j}$, $f_{j1}$ for $2\leq j\leq n$;
\item
then $\tilde{e}_{01}$, $\tilde{f}_{10}$ by $e_{00}$, $f_{00}$;
\item
then the remaining $\tilde{e}$'s, $\tilde{f}$'s by the remaining $e$'s, $f$'s;
\item
then  $\tilde{c}_{10}$, $\tilde{d}_{01}$,
$\tilde{c}_{0i}$, $\tilde{d}_{i0}$, $\tilde{c}_{j0}$, $\tilde{d}_{0j}$
by $d_{01}$, $c_{10}$, $c_{1i}$, $d_{i1}$, $c_{j0}$, $d_{0j}$ for
$1\leq i\leq n$ and $2\leq j\leq n$; 
\item
then $\tilde{c}_{00}$, $\tilde{d}_{00}$ by $c_{00}$, $d_{00}$;
\item
then the remaining $\tilde{c}$'s, $\tilde{d}$'s by the remaining $c$'s, $d$'s.
\end{itemize}
An inspection of the above identification shows that each of these maps is a tame isomorphism (in particular, each map involves only generators present at that moment), and so their composition is as well.
\end{proof}

\begin{proof}[Outline of proof of Theorem~\ref{thm:infty}]
Because of Theorem~\ref{thm:main}, it
suffices to show that the infinity transverse DGA over $R[U^{\pm
  1},V^{\pm 1}]$ is invariant under negative braid stabilization.

Let $B$ be an $n$-strand braid and $\tilde{B} = B\sigma_0^{-1}$ be its
negative stabilization, and let $(\alg=\CTinfty,\d^\infty)$ and
$(\tilde{\alg}=\widetilde{\mathit{CT}}^{\infty},\d^\infty)$ denote the infinity flavor of the version
of the transverse DGAs for $B$
and $\tilde{B}$, respectively, given as (\ref{DGA3}) in
Proposition~\ref{prop:modifiedDGA}. Write $\lambda' = \lambda
\mu^{-w(B)} (U/V)^{-(\sl(B)+1)/2}$.

Add to the DGA $(\alg,\d^\infty)$ the generators
\[
a_{0i},a_{i0},c_{00},c_{0i},c_{i0},d_{00},d_{0i},d_{i0},e_{00},e_{0i},f_{00},f_{i0},
\]
and extend the differential on $\alg$ to these generators by setting
$\d^\infty e_{00} = c_{00}$, $\d^\infty f_{00} = d_{00}$, $\d^\infty e_{0i} = d_{0i}$, $\d^\infty f_{i0} = c_{i0}$,
\[
\d^\infty c_{0i} = \mu U a_{0i}-\textstyle{\frac{\lambda'\mu U}{V}} \Phil_{1\ell} \ahat_{li},
\qquad
\d^\infty d_{i0} = V a_{i0} - \textstyle{\frac{V}{\lambda'\mu U}} \ahat_{i\ell} \Phir_{\ell 1}
\]
for $1\leq i\leq n$, and $\d^\infty=0$ for the other new generators. The result is a DGA that we now call $(\alg,\d^\infty)$ and is stable tame isomorphic to the original infinity DGA for $B$.

Now the following identification between $(\alg,\d^\infty)$ and $(\tilde{\alg},\tilde{\d}^\infty)$ is an isomorphism of DGAs (as usual, we leave this as an exercise):
\begin{alignat*}{2}
c_{00} &= \tilde{c}_{00} + \textstyle{\frac{\lambda'\mu U}{V}} \Phil_{1\ell} \tilde{d}_{\ell 0} & \qquad
d_{00} &= \tilde{d}_{00} + \textstyle{\frac{V}{\lambda'\mu U}} \tilde{c}_{0\ell} \Phir_{\ell 1} \\
c_{0i} &= \tilde{c}_{0i} &
d_{i0} &= \tilde{d}_{i0} \\
c_{10} &= -\tilde{c}_{10}+\tilde{d}_{11}-\tilde{c}_{1\ell}\Phir_{\ell 1} \phi(a_{10}) &
d_{01} &= -\tilde{d}_{01}+\tilde{c}_{11}-\phi(a_{01})\Phil_{1\ell}\tilde{d}_{\ell 1} \\
c_{1i} &= \tilde{c}_{1i}-\textstyle{\frac{1}{U}}\tilde{c}_{0i}+\textstyle{\frac{\lambda'}{V}} \tilde{d}_{00}\Phil_{1\ell}\acheck_{\ell i} &
d_{i1} &= \tilde{d}_{i1} - \textstyle{\frac{1}{V}}\tilde{d}_{i0} + \textstyle{\frac{1}{\lambda'\mu U}} \ahat_{i\ell} \Phir_{\ell 1} \tilde{c}_{00} \\
c_{j0} &= -\tilde{c}_{j0}+\tilde{d}_{j1}-V\tilde{c}_{j1} &
d_{0j} &= -\tilde{d}_{0j}+\tilde{c}_{1j}-U \tilde{d}_{1j} \\
& \qquad -\tilde{c}_{j\ell} \Phir_{\ell 1} \phi(a_{10}) - V
\Phil_{j\ell} \tilde{d}_{\ell 1} &&
\qquad -\phi(a_{01})\Phil_{1\ell}\tilde{d}_{\ell j}-U\tilde{c}_{1\ell}
\Phir_{\ell j} \\
c_{ji} &= \tilde{c}_{ji} &
d_{ij} &= \tilde{d}_{ij} \\
e_{00} &= \tilde{e}_{00} &
f_{00} &= \tilde{f}_{00} \\
e_{0i} &= \tilde{e}_{1i} &
f_{i0} &= \tilde{f}_{i1} \\
e_{11} &= -\textstyle{\frac{1}{U}}\tilde{e}_{01}+\tilde{e}_{11}-\textstyle{\frac{1}{\mu U}}\tilde{e}_{00} &
f_{11} &= -\textstyle{\frac{1}{V}}\tilde{f}_{10}+\tilde{f}_{11}-\textstyle{\frac{\mu}{V}} \tilde{f}_{00}
\\
& \qquad
-\textstyle{\frac{\lambda'}{V}} \tilde{d}_{00}\Phil_{1\ell}\tilde{d}_{\ell 1} + \textstyle{\frac{\lambda'\mu U}{V}} \tilde{f}_{00}\phi(a_{10}) &&
\qquad
+\textstyle{\frac{1}{\lambda'\mu U}}\tilde{c}_{1\ell}\Phir_{\ell 1}\tilde{c}_{00}+\textstyle{\frac{UV}{\lambda'\mu}} \phi(a_{01})\tilde{e}_{00} \\
e_{1j} &= -\textstyle{\frac{1}{U}} \tilde{e}_{0j}+\tilde{e}_{1j} -\textstyle{\frac{\lambda'}{V}} \tilde{d}_{00}\Phil_{1\ell}\tilde{d}_{\ell j} &
f_{j1} &= -\textstyle{\frac{1}{V}}\tilde{f}_{j0}+\tilde{f}_{j1} + \textstyle{\frac{1}{\lambda'\mu U}} \tilde{c}_{j\ell}\Phir_{\ell 1}\tilde{c}_{00} \\
e_{j_1j_2} &= \tilde{e}_{j_1j_2} &
f_{j_2j_1} &= \tilde{f}_{j_2j_1}.
\end{alignat*}
Here $1\leq i\leq n$, $2\leq j\leq n$, and $2\leq j_1\leq j_2\leq n$;
also, $\phi = \phi_B$, $\Phil = \Phil_B$, $\Phir = \Phir_B$, and all
expressions involving $\ell$ are summations over $1\leq\ell\leq n$. 
To see that the above identification represents a tame isomorphism, we note, as in the proof of Theorem~\ref{thm:main}, that it can be expressed as a composition of generator replacements beginning at $\tilde{\alg}$ and ending at $\alg$:
\begin{itemize}
\item
replace the $\tilde{a}$'s by the $a$'s;
\item
then
$\tilde{e}_{0i}$, $\tilde{f}_{i0}$ by $e_{1i}$, $f_{i1}$ for $1\leq i\leq n$;
\item
then the remaining $\tilde{e}$'s, $\tilde{f}$'s by the remaining $e$'s, $f$'s;
\item
then $\tilde{c}_{i0}$, $\tilde{d}_{0i}$ by $c_{i0}$, $d_{0i}$ for $1\leq i\leq n$;
\item
then $\tilde{c}_{1i}$, $\tilde{d}_{i1}$ by $c_{1i}$, $d_{i1}$ for $1\leq i\leq n$;
\item
then $\tilde{c}_{00}$, $\tilde{d}_{00}$ by $c_{00}$, $d_{00}$;
\item
then the remaining $\tilde{c}$'s, $\tilde{d}$'s by the remaining $c$'s, $d$'s.
\end{itemize}
An inspection of the above identification shows that each of these maps is a tame isomorphism, and so their composition is as well.
\end{proof}

\subsection{Comparison of transverse DGA conventions}
\label{ssec:eens}

In \cite{EENStransverse}, a version of the transverse DGA is presented
that differs slightly from the version given in this paper. The two
transverse DGAs agree up to signs and powers of $\lambda$ and $\mu$,
and are in fact stable tame isomorphic except for a shift of $\lambda$
to $-\lambda$ and $\mu$ to $\mu^{-1}$. Here we compare the conventions
and demonstrate this isomorphism. Although the transverse DGA in
\cite{EENStransverse}, which is written there as $(KC\alg^-,\d^-)$, is
defined for general transverse links, we will consider only the
single-component knot case for simplicity.

We recall the definition of the transverse DGA from
\cite{EENStransverse}, which we decorate with tildes where appropriate
for clarity of comparison. Thus for instance let $\tilde{\alg}_n$ be the
tensor algebra $R[U,V]$ generated by $n(n-1)$ formal variables
$\tilde{a}_{ij}$, $1\leq i\neq j\leq n$. Let
$\tilde{\phi} :\thinspace B_n \to \Aut\tilde{\alg}_n$ be the homomorphism
defined by
\[
\tilde{\phi}_{\sigma_k} :\thinspace
\begin{cases}
\tilde{a}_{ki} \mapsto \ta_{k+1,i}-\ta_{k+1,k}\ta_{ki}, & i\neq k,k+1 \\
\ta_{ik} \mapsto \ta_{i,k+1}-\ta_{ik}\ta_{k,k+1}, & i\neq k,k+1 \\
\ta_{k+1,i} \mapsto \ta_{ki}, & i\neq k,k+1 \\
\ta_{i,k+1} \mapsto \ta_{ik}, & i\neq k,k+1 \\
\ta_{k,k+1} \mapsto -\ta_{k+1,k} & \\
\ta_{k+1,k} \mapsto -\ta_{k,k+1} & \\
\ta_{ij} \mapsto \ta_{ij}, & i,j \neq k,k+1;
\end{cases}
\]
note that this agrees mod $2$ with the homomorphism $\phi$ introduced
in Section~\ref{sec:defs}. For $B\in B_n$, let $\tilde{\phi}_B$ be the
image of $B$ under this map, let $\tilde{\phi}_B^{\ext}$ be the
corresponding automorphism of $\tilde{\alg}_{n+1}$ obtained by including $B_n$
into $B_{n+1}$, and define matrices
$\tPhil_B(\tA),\tPhir_B(\tA)$ by
\begin{align*}
\tilde{\phi}_B^\ext(\ta_{i,n+1}) &= \sum_{\ell=1}^n (\tPhil_B(\tA))_{i\ell}
\ta_{\ell,n+1} \\
\tilde{\phi}_B^\ext(\ta_{n+1,i}) &= \sum_{\ell=1}^n \ta_{n+1,\ell}
(\tPhir_B(\tA))_{\ell i}.
\end{align*}
Define four more $n\times n$ matrices $\tAU,\tAV,\tBU,\tBV$ by
\begin{alignat*}{2}
(\tAU)_{ij} &= \begin{cases} \mu \ta_{ij} & i>j \\
\mu + U & i=j \\
U \ta_{ij} & i<j
\end{cases}
& \hspace{1in}
(\tAV)_{ij} &= \begin{cases} \mu V \ta_{ij} & i>j \\
1+\mu V & i=j \\
\ta_{ij} & i<j
\end{cases} \\
(\tBU)_{ij} &= \begin{cases} \mu \tb_{ij} & i>j \\
0 & i=j \\ U \tb_{ij} & i<j
\end{cases}
&
(\tBV)_{ij} &= \begin{cases} \mu V \tb_{ij} & i>j \\
0 & i=j \\
\tb_{ij} & i<j.
\end{cases}
\end{alignat*}
(Here $\tA,\tB$ have $2,0$ along the diagonal, respectively.)
The transverse DGA $(KC\alg^-,\td^-)$ defined in \cite{EENStransverse}
has generators $\{\ta_{ij}\}_{1\leq i\neq j\leq n}$ of degree $0$,
$\{\tb_{ij}\}_{1\leq i\neq j\leq n}$ and $\{\tilde{d}_{ij}\}_{1\leq
  i,j\leq n}$ of degree $1$, and $\{\tilde{f}_{ij}\}_{1\leq i,j\leq
  n}$ of degree $2$, with differential:
\begin{align*}
\td^-(\tA) &= 0, \\
\td^-(\tB) &=
- \Lam_B^{-1} \cdot \tA \cdot \Lam_B + \tilde{\phi}_B(\tA), \\
\td^-(\tD) &= \tAV \cdot \Lam_B + \tAU \cdot \tPhir_B(\tA), \\
\td^-(\tF) &=
\tBV \cdot(\tPhir_B(\tA))^{-1} + \tBU \cdot \Lam_B^{-1} -
\tPhil_B(\tA) \cdot \tD \cdot \Lam_B^{-1} + \Lam_B^{-1} \cdot \tD \cdot
(\tPhir_B(\tA))^{-1}.
\end{align*}

\begin{proposition}
The version of the transverse DGA from \cite{EENStransverse}
\label{prop:eens}
is stable
tame isomorphic to the version from this paper, once we replace
$\lambda$ by $-\lambda$ and $\mu$ by $\mu^{-1}$ in the latter.
\end{proposition}

We first present a lemma relating the two homomorphisms
$\phi,\tilde{\phi}$ and the matrices $\Phil,\Phir,\tPhil,\tPhir$.
Let $R :\thinspace B_n \to GL_n(\Z)$ be the representation given by
\[
R(\sigma_k)(e_i) = \begin{cases}
e_{k+1} & i=k \\
-e_k & i=k+1 \\
e_i & i\neq k,k+1
\end{cases}
\]
where $e_i$ is the usual ``basis'' vector in $\Z^n$ with $1$ in the
$i$-th coordinate and $0$ elsewhere.

\begin{lemma}
Set $\tA = -\A$; that is, $\ta_{ij} = -a_{ij}$ for all $i\neq j$. Then
\label{lem:eens}
\begin{align*}
\tPhil_B(\tA) &= \Del(v_B) \cdot \Phil_B(\A), \\
\tPhir_B(\tA) &= \Phir_B(\A) \cdot \Del(v_B)^{-1}, \\
\tilde{\phi}_B(\tA) &= - \Del(v_B) \cdot \phi_B(\A) \cdot \Del(v_B)^{-1},
\end{align*}
where $v_B = R(B^{-1})(e_1+\cdots+e_n)$ and $\Del(v_B)$ is the
diagonal $n\times n$ matrix with diagonal entries given by $v_B$.
\end{lemma}

\begin{proof}
The proofs of the expressions for $\tPhil_B(\tA)$ and $\tPhir_B(\tA)$
are straightforward inductions on the length of the braid word $B$.
For $\tPhil_B(\tA)$ (with a similar proof for $\tPhir_B(\tA)$), use
the chain rules
\[
\Phil_{B\sigma_k}(\A) = \Phil_{\sigma_k}(\phi_B(\A)) \cdot
\Phil_B(\A),
\qquad
\tPhil_{B\sigma_k}(\tA) =
\tPhil_{\sigma_k}(\tilde{\phi}_B(\tA)) \cdot \tPhil_B(\tA),
\]
cf.\ \cite[Prop.~4.4]{NgKCH1}, and the fact that
$\Phil_{\sigma_k}(\phi_B(\A))$
and $\tPhil_{\sigma_k}(\tilde{\phi}_B(\tA))$ are the identity matrix except
for the $2\times 2$ submatrices formed by rows $k,k+1$ and columns
$k,k+1$, where they are
$\left( \begin{smallmatrix} -\phi_B(a_{k+1,k}) & -1 \\ 1 & 0
\end{smallmatrix} \right)$ and
$\left( \begin{smallmatrix} -\tilde{\phi}_B(\ta_{k+1,k}) & 1 \\ 1 & 0
\end{smallmatrix} \right)$ respectively, cf.\ \cite[Lemma~4.6]{NgKCH1}.
The expression for $\tilde{\phi}_B(\tA)$ follows from the expressions
for $\tPhil_B(\tA)$ and $\tPhir_B(\tA)$, along with the identity
$\tilde{\phi}_B(\tA) = \tPhil_B(\tA) \cdot \tA \cdot \tPhir_B(\tA)$,
cf.\ Lemma~\ref{lem:philphir}.
\end{proof}

\begin{proof}[Proof of Proposition~\ref{prop:eens}]
Let $B$ be a braid whose closure is a knot, and let $\Del = \Del(v_B)$
be the diagonal matrix defined in Lemma~\ref{lem:eens}.
Let $(\tilde{\alg},\td^-)$ be the transverse DGA from
\cite{EENStransverse} as defined above; let $(\alg,\d^-)$ be the
version of our transverse DGA presented as (\ref{DGA2}) in
Proposition~\ref{prop:modifiedDGA} above, but with $\mu$ replaced by
$\mu^{-1}$ and $\Lam_B$ replaced
by $-\Lam_B \cdot \Del$:
\begin{align*}
\d^-(\A) &= 0 \\
\d^-(\B) &= \Del^{-1} \cdot \Lam_B^{-1} \cdot \A \cdot \Lam_B
\cdot \Del - \phi_B(\A) \\
\d^-(\D) &= -(\Acheck|_{\mu\mapsto\mu^{-1}}) \cdot \Lam_B \cdot \Del -
(\Ahat|_{\mu\mapsto\mu^{-1}})\cdot\Phir_B(\A) \\
\d^-(\F) &= (\check{\B}|_{\mu\mapsto\mu^{-1}}) \cdot (\Phir_B(\A))^{-1}
+ (\hat{\B}|_{\mu\mapsto\mu^{-1}}) \cdot \Del^{-1}\cdot \Lam_B^{-1} \\
& \qquad -
\Phil_B(\A) \cdot\D\cdot \Del^{-1}\cdot\Lam_B^{-1} +
\Del^{-1}\cdot\Lam_B^{-1}\cdot\D\cdot(\Phir_B)^{-1}.
\end{align*}
 Note that if $B$ has index $n$ and
writhe $w$, then $-\Del(v_B)$ has determinant $(-1)^{w+n} = -1$, and
so replacing $\Lam_B$ by $-\Lam_B \cdot \Del$ is equivalent to
replacing $\lambda$ by $-\lambda$, by Proposition~\ref{prop:lam}.

It thus suffices to show that $(\alg,\d^-)$ and $(\tilde{\alg},\td^-)$ are
tamely isomorphic. But under the tame isomorphism between $\alg$ and
$\tilde{\alg}$ given by
$\tA = -\A$, $\tB = \Del\cdot\B\cdot\Del^{-1}$, $\tD = \mu
\D\cdot\Del^{-1}$, $\tF = \mu\Del\cdot\F$, we have
\begin{alignat*}{2}
\tA^U &= -\mu(\Ahat|_{\mu\mapsto\mu^{-1}}) & \hspace{0.5in} \tA^V &=
-\mu(\Acheck|_{\mu\mapsto\mu^{-1}}) \\
\tB^U &=
\mu\Del\cdot(\hat{\B}|_{\mu\mapsto\mu^{-1}})\cdot\Del^{-1} &
\tB^V &= \mu \Del\cdot
(\check{\B}|_{\mu\mapsto\mu^{-1}})\cdot\Del^{-1};
\end{alignat*}
then Lemma~\ref{lem:eens}, along with the definitions of $\d^-$ and
$\td^-$, gives $\td^-(\tA)=\d^-(\tA)$, $\td^-(\tB)=\d^-(\tB)$,
$\td^-(\tD)=\d^-(\tD)$, and $\td^-(\tF)=\d^-(\tF)$.
\end{proof}

\section{Properties}
\label{sec:symmetries}

In this section, we present some basic properties of transverse homology.

\begin{proposition}
If $T$ is a destabilizable transverse knot (i.e., it is the
stabilization of another transverse knot), then the double-hat
transverse DGA of $T$ is trivial. \label{prop:stabdoublehat}
In particular,
\[
\HTdoublehat_*(T) = 0.
\]
\end{proposition}

\begin{proof}
As in the proof of Proposition~\ref{prop:HTinftyinv}, we calculate
that if $B\in B_n$ and $\tilde{B} = B\sigma_0^{-1}$ is a transverse
stabilization,
then we have
\[
\dminus(c_{00}) =
-1-\mu U-\lambda \mu^{-w(B)+1} V(\Phil_B)_{1\ell} a_{\ell 0}.
\]
Setting $U=V=0$ gives $\ddoublehat(c_{00}) = -1$.
\end{proof}

Note the formal similarity of Proposition~\ref{prop:stabdoublehat} to
the result of Chekanov \cite{Ch02} that the Legendrian DGA for a
stabilized Legendrian knot in $(\R^3,\xi_{\text{std}})$ is trivial.

We next examine various symmetries of the transverse DGA.

\begin{proposition}
If $\alpha$ is a unit, then we have a chain isomorphism of $R[U,V]$-algebras
\label{prop:rescale}
\[
(\CTminus_*(B;\lambda,\mu,U,V),\dminus) \cong
(\CTminus_*(B;\lambda
\alpha^{-\sl(B)},\mu/\alpha,U\alpha,V/\alpha),\dminus)
\]
with corresponding isomorphisms for $\CThat$ and $\CTdoublehat$, and a
chain isomorphism of $R[U^{\pm 1},V^{\pm 1}]$-algebras
\[
(\CTinfty_*(B;\lambda,\mu,U,V),\dinfty) \cong
(\CTinfty_*(B;\lambda\alpha,\mu/\alpha,U\alpha,V/\alpha),\dinfty).
\]
\end{proposition}

\begin{proof}
Replacing $(\lambda,\mu,U,V)$ by $(\lambda
\alpha^{-\sl(B)},\mu/\alpha,U\alpha,V/\alpha)$ keeps $\Ahat$
unchanged, multiplies $\Acheck$ by $\alpha^{-1}$, and replaces
$\Lam_B$ by
$\diag(\lambda\mu^{-w(B)}\alpha^{n(B)},1,\ldots,1)$. In the differential for the
transverse DGA, this has the same effect as keeping $\Ahat$ and
$\Acheck$ unchanged, replacing $\D$ by $\alpha \D$ (a tame isomorphism), and replacing
$\Lam_B$ by
$\diag(\lambda\mu^{-w(B)}\alpha^{n(B)-1},\alpha^{-1},\ldots,\alpha^{-1})$. Since
this matrix has the same determinant,
$\lambda\mu^{-w(B)}$, as $\Lam_B$, the result follows from Proposition~\ref{prop:lam}.
\end{proof}

It follows from Proposition~\ref{prop:rescale} that localizing at $V$ (allowing $V$ to be
invertible) in the transverse DGA is equivalent to setting $V=1$. In
particular, one can set $V=1$ in infinity transverse homology
without losing any information.

For our next symmetry, if $(\CTminus_*,\dminus)$ is a transverse
DGA, then we can construct another DGA $(\CTminus_*,\dminus_{\op})$
where $\dminus_{\op}$ is defined on generators of $\CTminus_*$ as
$\dminus$ but with every word reversed and appropriate signs
introduced. More precisely, on any graded tensor algebra over
$R[U,V]$, we can define an $R[U,V]$-module involution, $\op$, by
\[
\op(x_1x_2\cdots x_n) = (-1)^{\sum_{i<j} |x_i||x_j|} (x_n\cdots x_2x_1),
\]
where $x_1,x_2,\ldots,x_n$ are generators of the tensor algebra.
Now define $\dminus_{\op}$ on $\CTminus_*$ by
\[
\dminus_{\op} = \op \circ \dminus \circ \op.
\]
Since $\op(ab) = (-1)^{|a||b|}(\op b)(\op a)$ for any $a,b$, it is
easy to check that $\dminus_{\op}$ satisfies the Leibniz rule and
$(\dminus_{\op})^2=0$.

For $B$ a braid, the transverse complex $\CTminus_*(B)$ involves parameters $\lambda,\mu,U,V$; we make this explicit, where relevant, by writing $\CTminus_*(B;\lambda,\mu,U,V)$ for $\CTminus_*(B)$.

\begin{proposition}
We have a chain isomorphism of $R[U,V]$-algebras \label{prop:op}
\[
(\CTminus_*(B;\lambda^{-1},\mu^{-1},V,U),\dminus(B;\lambda^{-1},\mu^{-1},V,U))
\cong (\CTminus_*(B;\lambda,\mu,U,V),\dminus_{\op}(B;\lambda,\mu,U,V)).
\]
\end{proposition}

\noindent
In other words, one can switch the roles of $U$ and $V$ in the
transverse DGA, at the price of reversing word order and
inverting $\lambda$ and $\mu$.

We first establish a lemma.
Define an algebra isomorphism $\psi :\thinspace \CTminus_* \to \CTminus_*$ by
\begin{alignat*}{3}
\psi(\A) &= \A^T, & \quad
\psi(\B) &= \B^T, &\quad
\psi(\C) &= \mu \D^T, \\
\psi(\D) &= \mu \C^T, & \quad
\psi(\E) &= \mu \F^T, & \quad
\psi(\F) &= \mu \E^T,
\end{alignat*}
where $T$ represents transpose; that is, $\psi(a_{ij}) = a_{ji}$ and
so forth. Note that $\psi \circ \op = \op \circ \psi$. Extend $\psi$
and $\op$ to matrices as usual: $(\psi(\mathbf{M}))_{ij} =
\psi(\mathbf{M}_{ij})$, and similarly for $\op$.

\begin{lemma}
\begin{enumerate}
\item \label{op1}
If all of the entries in the matrix $\mathbf{M}_1$ (or all of the entries in $\mathbf{M}_2$) have degree $0$, then $(\mathbf{M}_1\mathbf{M}_2)^T = \op((\op\mathbf{M}_2)^T(\op\mathbf{M}_1)^T)$.
\item \label{op2}
$\Ahat|_{\mu\mapsto\mu^{-1}, U\mapsto V, V\mapsto U} = \mu^{-1} \psi(\Acheck^T)$ and
$\Acheck|_{\mu\mapsto\mu^{-1}, U\mapsto V, V\mapsto U} = \mu^{-1} \psi(\Ahat^T)$.
\item \label{op3}
$\psi \op \Phil_B = (\Phir_B)^T$ and
$\psi \op \Phir_B = (\Phil_B)^T$.
\end{enumerate}
\label{lem:op}
\end{lemma}

\begin{proof}
(\ref{op1}) is immediate from the definition of $\op$, and (\ref{op2}) is immediate from the definition of $\Ahat$ and $\Acheck$. (\ref{op3}) can be seen from the fact that the homomorphism $\phi$ satisfies
$\psi\phi(a_{ij}) = \op\phi(a_{ji})$ for all $i,j$ by the construction of $\phi$.
\end{proof}

\begin{proof}[Proof of Proposition~\ref{prop:op}]
For clarity, denote the differential $\d^-$ on $\CTminus_*$ with $(\lambda,\mu,U,V)$ replaced by $(\lambda^{-1},\mu^{-1},V,U)$ by $\tilde{\d}^-$. We claim that $\psi\circ\d^-_{\op} = \tilde{\d}^-\circ\psi$.

This is a fairly routine claim to verify, given Lemma~\ref{lem:op}; we will check that $\psi\circ\d^-_{\op}(\C) = \tilde{\d}^-\circ\psi(\C)$, and leave verification for the other generators of $\CTminus_*$ to the reader:
\begin{align*}
\tilde{\d}^-(\psi(\C)) &=
\mu \tilde{\d}^-(\D^T) \\
&= \mu (\Acheck^T|_{\mu\mapsto\mu^{-1}, U\mapsto V, V\mapsto U}) - \mu ((\Ahat|_{\mu\mapsto\mu^{-1}, U\mapsto V, V\mapsto U}) \cdot \Phir_B\cdot \Lam_B)^T \\
&= \psi(\Ahat) - \op(\Lam_B \cdot\op(\Phir_B)^T \cdot\op\psi(\Acheck)) \\
&=
\psi(\Ahat - \op(\Lam_B \cdot \Phil_B \cdot \Acheck)) \\
&= \psi(\d^-_{\op}(\C)).
\end{align*}
The proposition follows.
\end{proof}

We can use Proposition~\ref{prop:op} to understand the effect on the
transverse DGA of an
operation on transverse knots known as mirroring.

\begin{definition}
Let $T$ be a transverse knot represented by the closure of a braid
$B$. The \textit{transverse mirror} of $T$, written $\mir(T)$, is the
transverse knot represented by the closure of the braid obtained by
reversing the order of the letters in $B$.
\end{definition}

\noindent
Note that in the topological category, transverse mirroring takes a
knot to its orientation reverse.
In \cite{NgThurston}, transverse mirrors are defined in terms of Legendrian
approximations: the transverse mirror of a Legendrian approximation
$\Lambda$ to a transverse knot is the transverse pushoff of
$-\mu(\Lambda)$, the orientation reverse of the Legendrian mirror of
$\Lambda$. It is an easy exercise (cf.\ \cite{KhN}) to check that this
definition agrees with ours.

One example of transverse mirrors is the family of pairs of $3$-braids
considered by Birman and Menasco in \cite{BM3braid}:
$\sigma_1^u\sigma_2^v\sigma_1^w\sigma_2^{-1}$ and
$\sigma_1^w\sigma_2^v\sigma_1^u\sigma_2^{-1}$, which are related by a
negative flype and thus topologically isotopic. The transverse mirror
of the first of these is
$\sigma_2^{-1}\sigma_1^w\sigma_2^v\sigma_1^u$, which is conjugate and
thus transversely isotopic to the second.

For the next result, we borrow a definition from \cite{NgKCH1}: for $B
\in B_n$, define $B^*$ to be the image of
$B$ under the group homomorphism on $B_n$ sending $\sigma_k$ to
$\sigma_{n-k}^{-1}$ for all $k$; then the braids $B$ and $(B^*)^{-1}$
represent transverse mirrors.

\begin{proposition}
We have a chain isomorphism of $R[U,V]$-algebras
\[
(\CTminus_*((B^*)^{-1};\lambda,\mu,U,V),\dminus((B^*)^{-1};\lambda,\mu,U,V)
\cong (\CTminus_*(B;\lambda,\mu,U,V),\dminus_{\op}(B;\lambda,\mu,U,V)).
\]
In other words, if $T$ is a transverse knot, then the transverse DGA
of the transverse mirror of $T$ is the opposite of the
transverse DGA of $T$.
\label{prop:mirror}
\end{proposition}

\begin{proof}
We follow similar proofs from \cite{NgKCH1,Ngframed}. Let $B$ be an
$n$-strand braid and let $\CTminus_*$ denote the usual algebra
generated by $a,b,c,d,e,f$ generators.
By
Proposition~\ref{prop:op}, it suffices to give an isomorphism
between
$(\CTminus_*,\dminus((B^*)^{-1};\lambda^{-1},\mu^{-1},V,U)$ and
$(\CTminus_*,\dminus(B;\lambda,\mu,U,V))$. For clarity, write
$\tilde{\d}^-$ for $\dminus((B^*)^{-1};\lambda^{-1},\mu^{-1},V,U)$ as in
Definition~\ref{def:transDGA}, and write $\d^-$ for
$\dminus(B;\lambda,\mu,U,V)$ as in Definition~\ref{def:transDGA} but
with $\Lam_B$ replaced by
$\diag(1,\ldots,1,\lambda\mu^{-w(B)})$. By Proposition~\ref{prop:lam},
it suffices to show that $(\CTminus_*,\d^-)$ and
$(\CTminus_*,\tilde{\d}^-)$ are chain isomorphic.

Let $\Xi$ be the operation on $n\times n$ matrices defined by
$\Xi(\mathbf{M})_{ij} = \mathbf{M}_{n+1-i,n+1-j}$, i.e., $\Xi$
conjugates by the $n\times n$ matrix with $1$'s on the NE-SW diagonal and $0$'s
everywhere else. Define an algebra isomorphism $\xi :\thinspace \CTminus \to
\CTminus$ by
\begin{alignat*}{2}
\xi(\A) &= \Xi(\phi_B(\A)), &
\qquad \xi(\B) &= -\Lam_B^{-1}\cdot\Xi(\B)\cdot\Lam_B, \\
\xi(\C) &= -\mu^{-1}\Lam_B^{-1}\cdot \Xi(\C\cdot\Phir_B(\A)), &
\xi(\D) &= -\mu^{-1} \Xi(\Phil_B(\A)\cdot\D)\cdot \Lam_B, \\
\xi(\E) &= -\mu^{-1} \Lam_B^{-1} \cdot \Xi(\F) \cdot\Lam_B, &
\xi(\F) &= -\mu^{-1} \Lam_B^{-1} \cdot \Xi(\E) \cdot\Lam_B.
\end{alignat*}
We claim that $\xi \circ \tilde{\d}^- = \d^-\circ\xi$.

From \cite[Cor.~4.5]{NgKCH1} and the proof of \cite[Prop.~6.9]{NgKCH1}, we have
\[
\xi(\Phil_{(B^*)^{-1}}(\A)) = \Xi(\Phil_{B^{-1}}(\phi_B(\A))) =
\Xi(\Phil_B(\A))^{-1},
\]
and similarly $\xi(\Phir_{(B^*)^{-1}}(\A)) = \Xi(\Phir_B(\A))^{-1}$. In
addition, one readily checks from the definitions that
\begin{align*}
\xi(\Ahat|_{\mu\mapsto\mu^{-1}, U\mapsto V, V\mapsto U}) &= \mu^{-1}
\Xi(\phi_B(\Acheck)) \\
\xi(\Acheck|_{\mu\mapsto\mu^{-1}, U\mapsto V, V\mapsto U}) &= \mu^{-1}
\Xi(\phi_B(\Ahat)).
\end{align*}
With these identities in hand, it is straightforward to prove the
claim; we check that $\xi \circ \tilde{\d}^-(\C) = \d^-\circ\xi(\C)$,
and leave verification for the other generators of $\CTminus_*$ to the
reader:
\begin{align*}
\xi(\tilde{\d}^-(\C)) &=
\xi\left((\Ahat|_{\mu\mapsto\mu^{-1}, U\mapsto V, V\mapsto U}) -
\Lam_B^{-1} \cdot \Phil_{(B^*)^{-1}}(\A) \cdot
(\Acheck|_{\mu\mapsto\mu^{-1}, U\mapsto V, V\mapsto U}) \right) \\
&= \mu^{-1} \Xi(\phi_B(\Acheck)) - \mu^{-1} \Lam_B^{-1} \cdot
\Xi(\Phil_B(\A))^{-1} \cdot \Xi(\phi_B(\Ahat)) \\
&= -\mu^{-1}\Lam_B^{-1} \Xi\left(
(\Ahat-\Xi(\Lam_B)\cdot\Phil_B(\A)\cdot\Acheck)\cdot\Phir_B(\A)\right)
\\
&= \d^-(\xi(\C)).
\end{align*}
The proposition follows.
\end{proof}

Like Proposition~\ref{prop:stabdoublehat}, Proposition~\ref{prop:mirror} may be compared to an analogous result in Legendrian contact homology, namely that the Chekanov--Eliashberg DGA for the Legendrian mirror of a Legendrian knot $\Lambda$ is the opposite of the DGA for $\Lambda$. See \cite{NgCLI}.

\section{Computations and Applications}
\label{sec:applications}

In this section, we perform some calculations with transverse homology
and demonstrate that it constitutes an effective invariant of
transverse knots.

\subsection{The augmentation polynomial of a topological knot}

The infinity transverse homology of a topological knot, a graded
algebra over $R[U^{\pm 1},V^{\pm 1}]=\Z[\lambda^{\pm 1},\mu^{\pm
  1},U^{\pm 1},V^{\pm 1}]$, is a somewhat unwieldy object
to compute. There are various ways to extract information out of
$\HTinfty$; here we highlight one, a three-variable polynomial that is
in some sense a generalization of the two-variable $A$-polynomial.

\begin{definition}
Let $K$ be a topological knot, and let
$\lambda_0,\mu_0,U_0,V_0\in\CC^*=\CC\setminus\{0\}$. We say that the DGA
$(\CTinfty_*(K),\dinfty)$ \textit{has an augmentation at}
$(\lambda_0,\mu_0,U_0,V_0)$ if there is a $\CC$-algebra map
\[
\HTinfty_0(K) \otimes_{R[U^{\pm 1},V^{\pm 1}]} \CC \to \CC,
\]
where $\CC$ is viewed as an $R[U^{\pm 1},V^{\pm 1}]$-algebra with
$\lambda,\mu,U,V$ acting as multiplication by the complex scalars
$\lambda_0,\mu_0,U_0,V_0$, respectively.

The \textit{augmentation variety} of $K$ is the set
\[
\{(\lambda_0,\mu_0,U_0) \, : \, (\CTinfty_*(K),\dinfty) \text{ has an
  augmentation at } (\lambda_0,\mu_0,U_0,1)\} \subset (\CC^*)^3.
\]
If the augmentation
variety is not complex $3$-dimensional, then the union of
its $2$-dimensional components is the zero set of the
\textit{augmentation polynomial} $\Aug_K(\lambda,\mu,U) \in
\CC[\lambda,\mu,U]$, well-defined up to constant multiplication if we
specify that it contains no repeated factors and is not divisible by
$\lambda$, $\mu$, or $U$.
\end{definition}

Note that the $V$ coordinate has been dropped from the definition of
the augmentation variety and polynomial; this is because of
Proposition~\ref{prop:rescale}, which implies that the $V$ information
is superfluous in the infinity theory.

\begin{example}
For the unknot $U$, the computation from
Example~\ref{ex:unknot} shows that $\Aug_U(\lambda,\mu,U) = -1-\mu U+\lambda
+\lambda\mu$.
\end{example}

\begin{example}
Consider the right-handed trefoil $T$ given as the closure of $\sigma_1^3
\in B_2$. The infinity transverse DGA has $\dinfty(b_{21}) =
a_{21}-\frac{\mu^3 U}{\lambda V}a_{12}$, and so $a_{21} = \frac{\mu^3
  U}{\lambda V} a_{12}$ in $\HTinfty_0(T)$. It follows that
$\HTinfty_0(T)$ is a quotient of the polynomial ring $(R[U^{\pm
  1},V^{\pm 1}])[x]$ with $x=a_{12}$. A computation of
$\dinfty(\C),\dinfty(\D)$ for the trefoil yields
\begin{align*}
\HTinfty_0(T) &\cong (\Z[\lambda^{\pm 1},\mu^{\pm 1},U^{\pm 1},V^{\pm
  1}])[x] \, / \\
& \qquad (\mu^4 U x^2+\mu^3 U x-\lambda\mu V-\lambda V^2,
\lambda\mu^2 V x-\mu^3 U x+\lambda V^2-\lambda\mu^2 U V).
\end{align*}
Up to multiplication by an overall unit and after setting $V=1$, the
resultant of the two
polynomials above in $x$, and thus the augmentation polynomial of $T$,
is
\[
\Aug_T(\lambda,\mu,U) =
(\lambda\mu^4-\mu^4)U^3+(\lambda\mu^3-\mu^3-2\lambda\mu^2) U^2+(2\lambda\mu^2+\lambda\mu+\lambda)U+(-\lambda^2\mu-\lambda^2).
\]
\end{example}

For any knot, since transverse homology descends to the knot contact
homology from \cite{Ngframed} when we set $U=V=1$, the intersection of
the augmentation variety in $(\CC^*)^3$ with the plane $U=1$ yields
the augmentation variety in $(\CC^*)^2$ considered in
\cite{Ngframed}. One might reasonably guess that the specialization
$\Aug_K(\lambda,\mu,1)$ should often equal the augmentation
polynomial $\tilde{A}_K(\lambda,\mu)$ from \cite{Ngframed}, whose
vanishing set is the augmentation variety in $(\CC^*)^2$. This is true
for the unknot and the trefoil; we do not know if it is true in
general.

From \cite{Ngframed}, $\tilde{A}_K(\lambda,\mu)$ contains the
$A$-polynomial of $K$ as a factor. One might then consider
$\Aug_K(\lambda,\mu,U)$ to be some sort of three-variable
generalization of the $A$-polynomial. A geometric interpretation for
the three-variable augmentation polynomial (in
terms of representations of $\pi_1(\R^3\setminus K)$) is currently not
known to the author.

\subsection{Transverse computations}

Here we provide evidence that transverse homology is quite effective
as an invariant of transverse knots.
More precisely, we will see that $\HThat_0$,
considered as an algebra over $R$, can be used to distinguish pairs of
transverse knots with the same classical invariants. (For unknown
reasons, it appears in examples that $\HThat_0$ is more effective at
distinguishing transverse knots than $\HTdoublehat_0$.) As an easily
implemented computational tool, we use augmentation numbers in a
similar manner to \cite{NgKCH1,Ngframed}.

\begin{definition}
Let $\kk$ be a finite field and let $\lambda_0,\mu_0$ be nonzero elements of $\kk$. For a finitely generated $R$-algebra $A$, the \textit{augmentation number} $\Aug(A,\kk,\lambda_0,\mu_0)$ is the number of $\kk$-algebra maps
$A \otimes_{\Z} \kk \to \kk$ sending $\lambda\in R$ to $\lambda_0$ and
$\mu\in R$ to $\mu_0$.
\end{definition}

\noindent
We have the following corollary of Theorem~\ref{thm:infty}.

\begin{proposition}
If $T_1,T_2$ are transverse knots and there exist
$\kk,\lambda_0,\mu_0$ for which
$\Aug(\HThat_0(T_1),\kk,\lambda_0,\mu_0) \neq
\Aug(\HThat_0(T_2),\kk,\lambda_0,\mu_0)$, then $T_1,T_2$ are not
transversely isotopic.
\label{prop:aug}
\end{proposition}

In practice, augmentation numbers for finitely generated, finitely
presented $R$-algebras are straightforward to calculate by
computer.  Our computations rely on
\texttt{transverse.m}, and executable \textit{Mathematica} notebooks
producing these computations are available at the author's web site.

On the next two pages, we present a table of the $13$ knot types of
arc index $9$ or fewer that are suggested by the work of \cite{Atlas}
to be transversely nonsimple; this could be conjectured to be a
complete list of such knot types (transversely nonsimple with arc index at most $9$). For each type, the table includes the guess from \cite{Atlas} for all of the distinct transverse knots of that type with the
relevant self-linking number. Of these $13$ topological knots, $6$
have been previously shown to be transversely nonsimple, all via the
transverse invariant in Heegaard Floer homology; see
\cite{Atlas,NOT,OStipsicz}.

By contrast, at least $10$ of the $13$ knots can be shown to be
transversely nonsimple by transverse homology and
Proposition~\ref{prop:aug}, including all $6$ that
can be distinguished by Heegaard Floer homology.
Indeed, as the following table shows, it suffices to compute
augmentation numbers for as small a field as $\kk = \Z/3$.\footnote{The computer program has difficulty computing augmentation numbers for braids with more than $4$ strands. To compute the augmentation numbers for the non-stabilized $5$-strand transverse representatives of $m(10_{145})$ and $12n_{591}$ in the table, we used a slightly different version of degree $0$ transverse homology than Definition~\ref{def:deg0}: it is easy to prove that if $B \in B_n$ is a product of two braids $B=B_1B_2$, then
\[
\mathit{HT}_0(B) \cong \alg_n\,/\,(\Phil_{B_1^{-1}}\cdot\Ahat - \Lam_B\cdot\Phil_{B_2}\cdot\Acheck,\,
\Acheck\cdot\Phir_{B_1^{-1}}-\Ahat\cdot\Phir_{B_2}\cdot\Lam_B^{-1}).
\]
One can use this formulation of $\mathit{HT}_0$ to eliminate enough generators of $\alg_n$ to allow for the computation of augmentation numbers for $m(10_{145})$ and $12n_{591}$. For the stabilized representatives of $m(10_{145})$ and $12n_{591}$, the computation is much easier: an argument along the lines of Proposition~\ref{prop:stabdoublehat} shows that any augmentation number $\Aug(\widehat{\mathit{HT}}_0(T),\mathbf{k},\lambda_0,\mu_0)$ for a stabilized transverse knot must be $0$ unless $\mu_0 = -1$.}

One particular knot of interest is the twist knot $m(7_2)$. Here
Ozsv\'ath and Stipsicz \cite{OStipsicz} have shown transverse nonsimplicity
using the Heegaard Floer LOSS invariant, along with a naturality
argument. (By contrast, the other $5$ knots that can be distinguished
by Heegaard Floer can all be treated by a computer program.)
Transverse homology, on the other hand, can be used to distinguish the
$m(7_2)$ knots by computer, without any geometric input.

The remaining $3$ knots of arc index $\leq 9$ that are conjectured to
be transversely nonsimple are all of the following form: the pair of
transverse knots that are conjecturally transversely nonisotopic are
related by the transverse mirror operation. Since augmentation numbers
depend only on the abelianization of the DGA, and the transverse DGAs
for transverse mirrors have the same abelianization by
Proposition~\ref{prop:mirror}, augmentation numbers will never be able
to distinguish mirrors. It is not inconceivable that the full
transverse DGA might sometimes distinguish mirrors, but this seems to
be a very delicate point. We remark that the (grid-diagram) Heegaard
Floer transverse 
invariant cannot currently distinguish transverse
mirrors, at least via the techniques of \cite{NOT}; see 
\cite[Prop.~1.2]{OST}. (It is probable that this will change in the
future with the advent of naturality results.)

In the table, for each knot type, transverse representatives that are
(conjecturally) distinct are depicted in two ways: grid diagrams
corresponding to Legendrian approximations (as taken from
\cite{Atlas}) and braid closures. (For the algorithm to get from one
to the other, see, e.g., \cite{KhN}. For braids, numbers represent
braid generators and bars represent inverses; e.g.,
$3\,3\,\overline{2}\,3\,2\,1\,1\,2\,\overline{1}$ is the $4$-braid
$\sigma_3^2\sigma_2^{-1}\sigma_3\sigma_2\sigma_1^2\sigma_2\sigma_1^{-1}$.)
The table then indicates whether
the two transverse invariants, the Heegaard Floer invariant and
transverse homology, can distinguish the transverse
representatives. In the case of a positive answer for transverse
homology, the relevant augmentation-number computation is given. Of
the $7$ knot types with a non-positive answer for the transverse Heegaard
Floer invariant, $5$ have $\widehat{\mathit{HFK}}=0$ in the relevant
bidegree, while the other $2$, $m(9_{45})$ and $10_{128}$, are
transverse mirrors. (The same result holds for Khovanov homology in
the relevant bidegree and Khovanov--Rozansky homology in the relevant
triple degree, indicating that the transverse invariants of
Plamenevskaya \cite{PlamKh} and Wu \cite{WuKR}, like the Heegaard Floer
invariants, do not distinguish these transverse knots.) 

\begin{center}
\begin{tabular}{|c||c|c|c|c|c|}
\hline
Knot & Grid & Braid & $\mathit{HFK}$? & $\mathit{HT}$? & $\mathit{HT}$ computation \\
\hline
\multirow{2}{*}{$m(7_2)$}
& \includegraphics[height=0.5in]{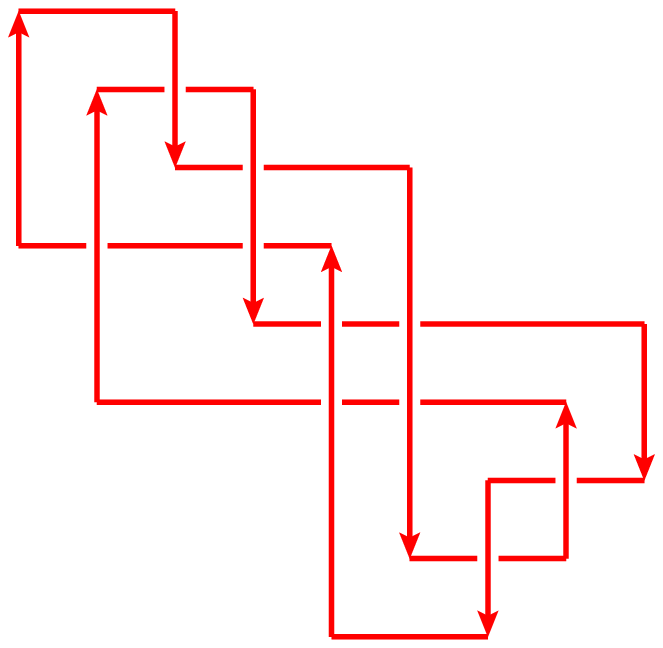} &
\rb{$3\,3\,\overline{2}\,3\,2\,1\,1\,2\,\overline{1}$}
&
\multirow{2}{*}{$\checkmark$ \cite{OStipsicz}} & \multirow{2}{*}{$\checkmark$} &
\rb{$\Aug(\HThat_0,\Z/3,2,1)=0$} \\
\cline{2-3} \cline{6-6}
& \includegraphics[height=0.5in]{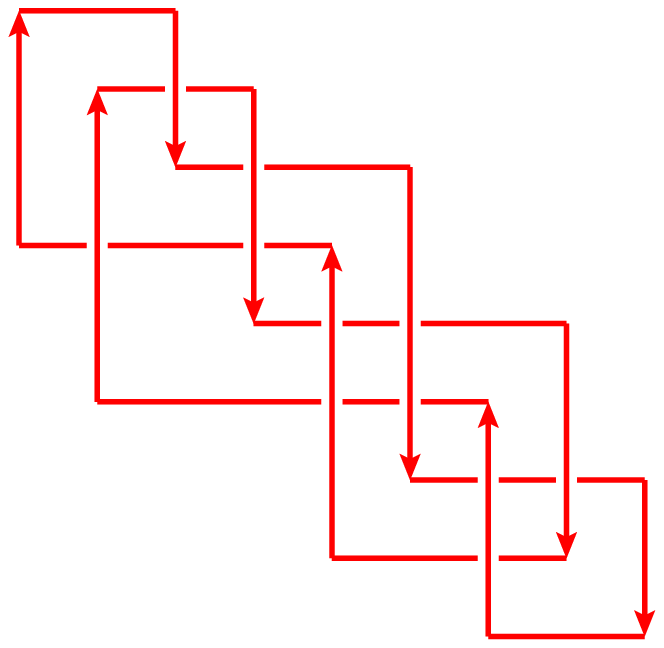} &
\rb{$3\,3\,\overline{2}\,3\,2\,\overline{1}\,2\,1\,1$}
&&&
\rb{$\Aug(\HThat_0,\Z/3,2,1)=5$} \\
\hline
\multirow{2}{*}{$m(7_6)$}
& \includegraphics[height=0.5in]{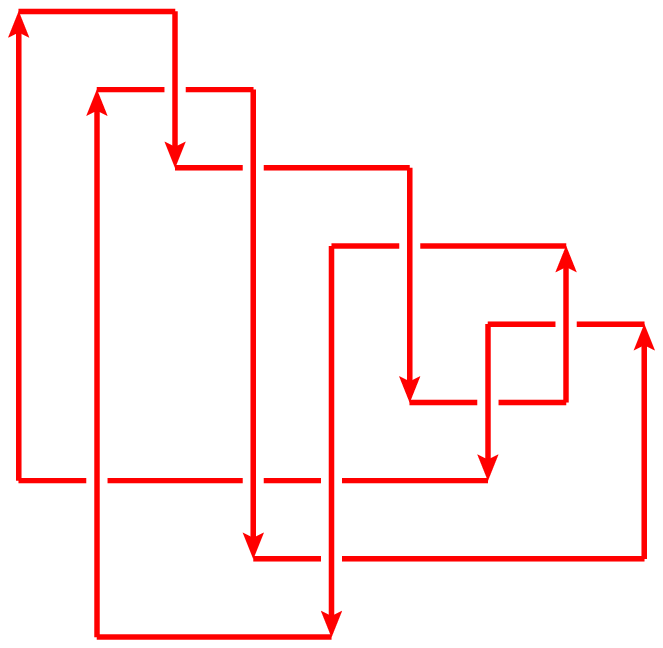} &
\rb{$1\,\overline{2}\,1\,\overline{2}\,\overline{3}\,2\,3\,3\,3$}
&
\multirow{2}{*}{\ding{55}} & \multirow{2}{*}{$\checkmark$} &
\rb{$\Aug(\HThat_0,\Z/3,2,1)=5$} \\
\cline{2-3} \cline{6-6}
& \includegraphics[height=0.5in]{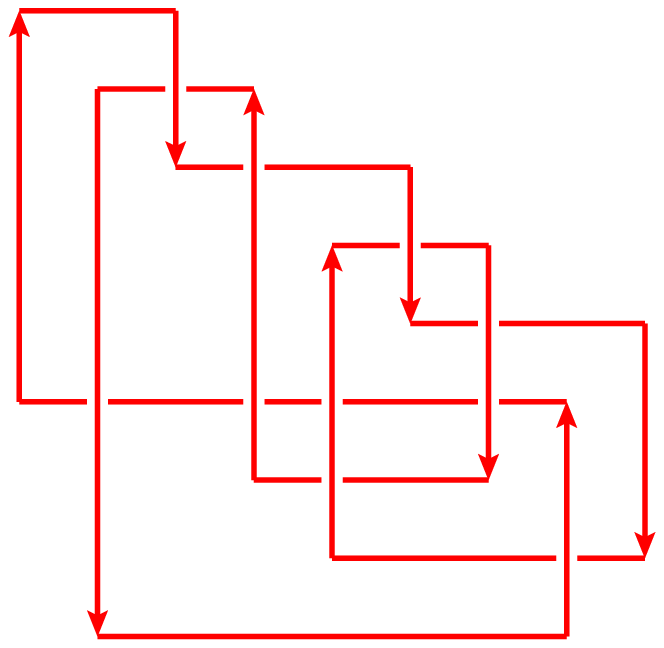} &
\rb{$1\,\overline{2}\,1\,\overline{2}\,3\,3\,3\,2\,\overline{3}$}
&&&
\rb{$\Aug(\HThat_0,\Z/3,2,1)=0$} \\
\hline
\multirow{3}{*}{$9_{44}$}
& \includegraphics[height=0.5in]{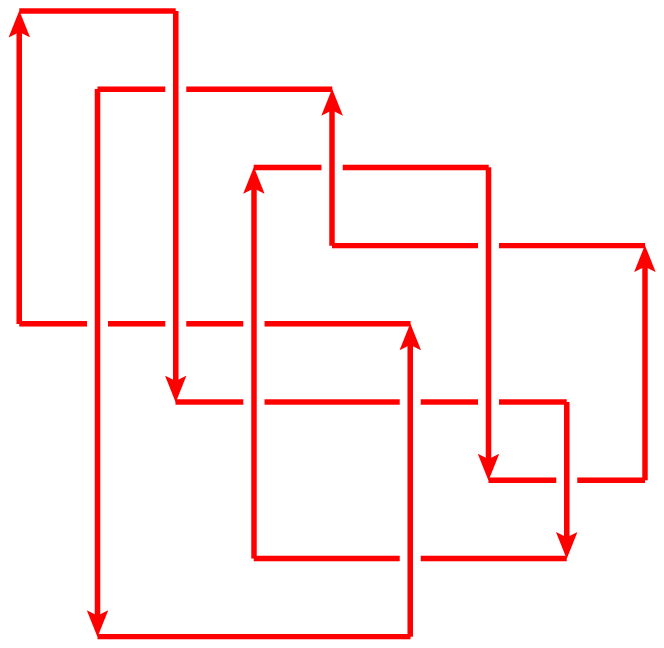} &
\rb{$\overline{3}\,1\,2\,\overline{3}\,\overline{2}\,3\,1\,\overline{2}\,\overline{3}$}
&
\multirow{3}{*}{\ding{55}} & \multirow{3}{*}{$\checkmark$/?} &
\rb{$\Aug(\HThat_0,\Z/3,2,1)=5$} \\
\cline{2-3} \cline{6-6}
& \includegraphics[height=0.5in]{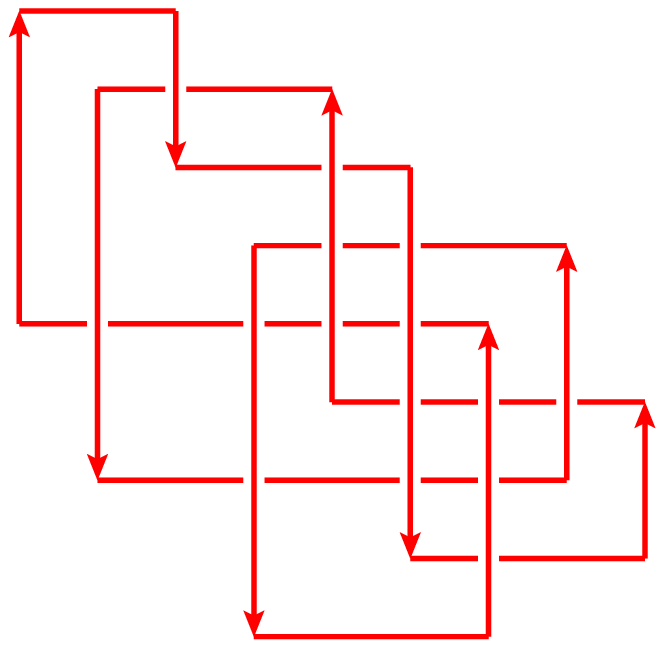} &
\rb{$\overline{2}\,\overline{3}\,2\,1\,2\,\overline{3}\,\overline{2}\,1\,\overline{2}$}
&&&
\rb{$\Aug(\HThat_0,\Z/3,2,1)=0$} \\
\cline{2-3} \cline{6-6}
& \raisebox{0.25in}{\includegraphics[height=0.5in,angle=180]{9_44tb.eps}} &
(mirror of previous braid)
&&&
$\Aug(\HThat_0,\Z/3,2,1)=0$ \\
\hline
\multirow{2}{*}{$m(9_{45})$}
& \includegraphics[height=0.5in]{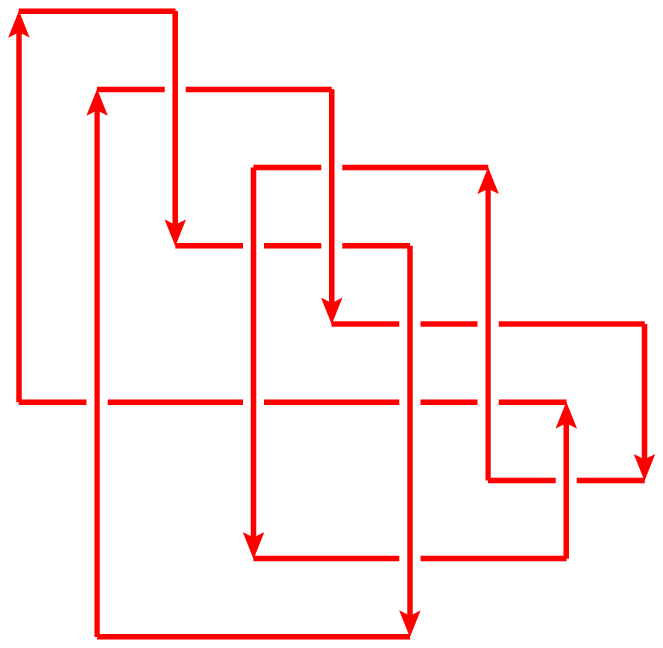} &
\rb{$2\,\overline{3}\,2\,1\,3\,\overline{2}\,3\,1\,2$}
&
\multirow{2}{*}{?} & \multirow{2}{*}{?} &
\rb{---} \\
\cline{2-3} \cline{6-6}
& \raisebox{0.25in}{\includegraphics[height=0.5in,angle=180]{9_44tb.eps}} &
(mirror of previous braid)
&&&
---  \\
\hline
\multirow{2}{*}{$9_{48}$}
& \includegraphics[height=0.5in]{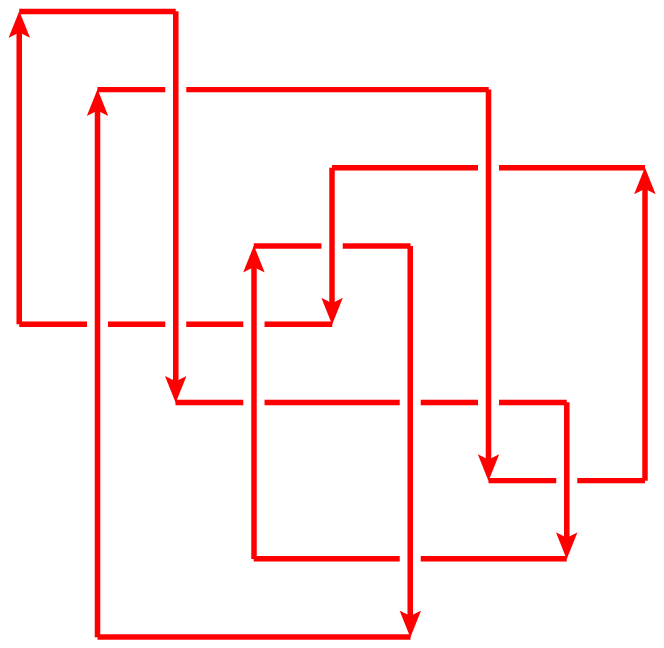} &
\rb{$\overline{2}\,3\,3\,2\,\overline{1}\,2\,\overline{3}\,2\,1\,1\,\overline{2}$}
&
\multirow{2}{*}{\ding{55}} & \multirow{2}{*}{$\checkmark$} &
\rb{$\Aug(\HThat_0,\Z/3,2,1)=4$} \\
\cline{2-3} \cline{6-6}
& \includegraphics[height=0.5in]{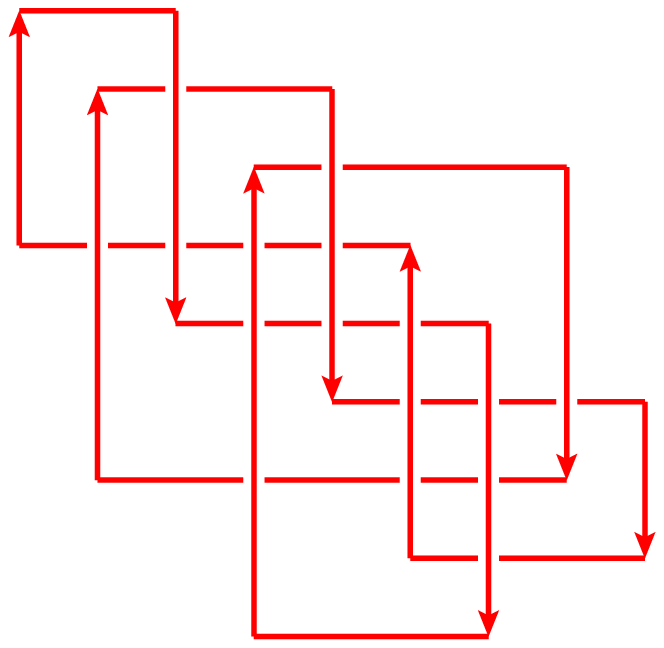} &
\rb{$2\,3\,3\,2\,\overline{1}\,\overline{2}\,\overline{2}\,\overline{3}\,2\,1\,1$}
&&&
\rb{$\Aug(\HThat_0,\Z/3,2,1)=0$} \\
\hline
\multirow{2}{*}{$10_{128}$}
& \includegraphics[height=0.5in]{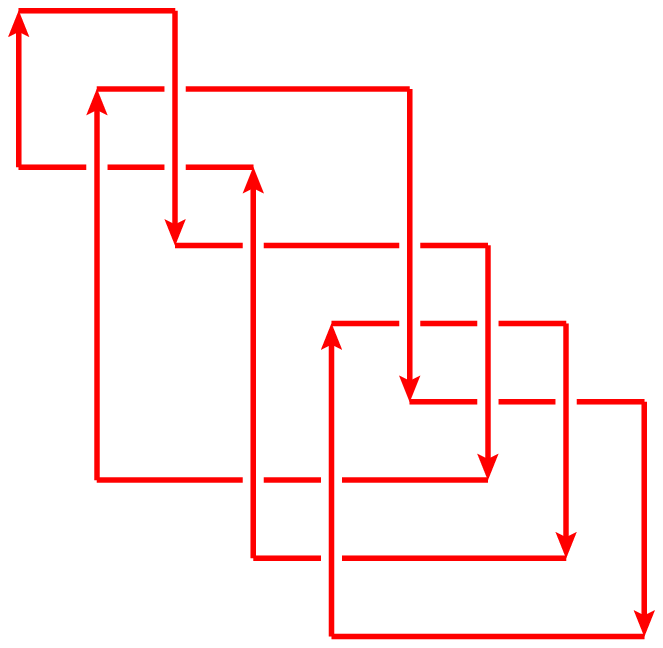} &
\rb{$1\,2\,1\,2\,1\,2\,1\,3\,3\,2\,\overline{3}$}
&
\multirow{2}{*}{?} & \multirow{2}{*}{?} &
\rb{---} \\
\cline{2-3} \cline{6-6}
& \raisebox{0.25in}{\includegraphics[height=0.5in,angle=180]{10_128b.eps}} &
(mirror of previous braid)
&&&
---  \\
\hline
\end{tabular}
\end{center}

\begin{center}
\begin{tabular}{|c||c|c|c|c|c|}
\hline
Knot & Grid & Braid & $\mathit{HFK}$? & $\mathit{HT}$? & $\mathit{HT}$ computation \\
\hline
\multirow{2}{*}{$m(10_{132})$}
& \includegraphics[height=0.5in]{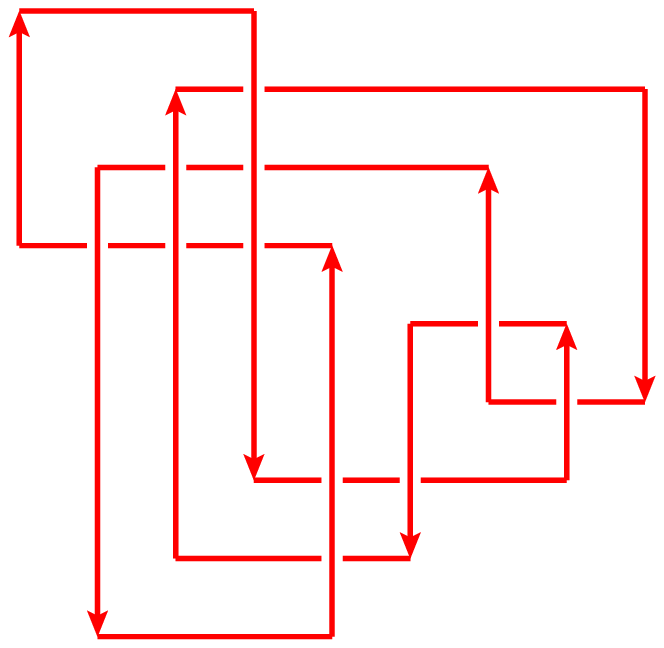} &
\rb{$3\,\overline{2}\,\overline{2}\,3\,3\,2\,\overline{3}\,\overline{1}\,2\,1\,1$}
&
$\checkmark$ & \multirow{2}{*}{$\checkmark$} &
\rb{$\Aug(\HThat_0,\Z/3,1,1)=0$} \\
\cline{2-3} \cline{6-6}
& \includegraphics[height=0.5in]{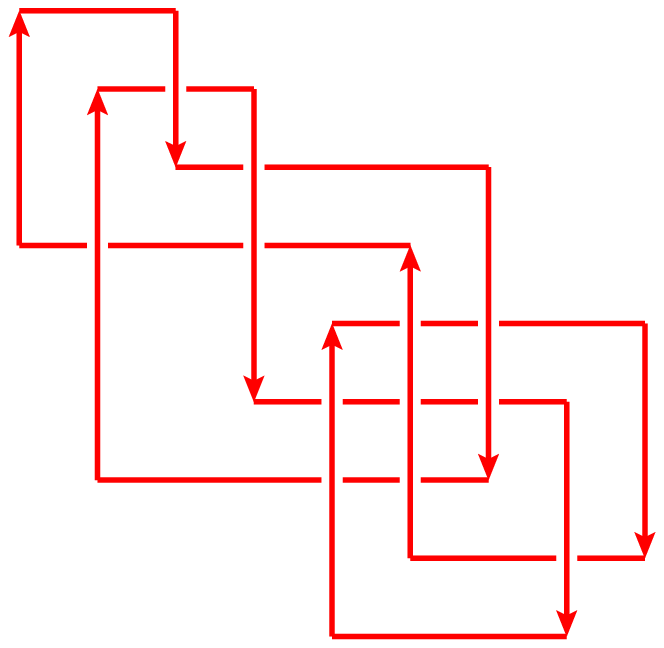} &
\rb{$3\,\overline{2}\,\overline{2}\,3\,3\,2\,\overline{3}\,1\,1\,2\,\overline{1}$}
&\raisebox{0.3in}{\cite{NOT}} &&
\rb{$\Aug(\HThat_0,\Z/3,1,1)=1$} \\
\hline
\multirow{2}{*}{$10_{136}$}
& \includegraphics[height=0.5in]{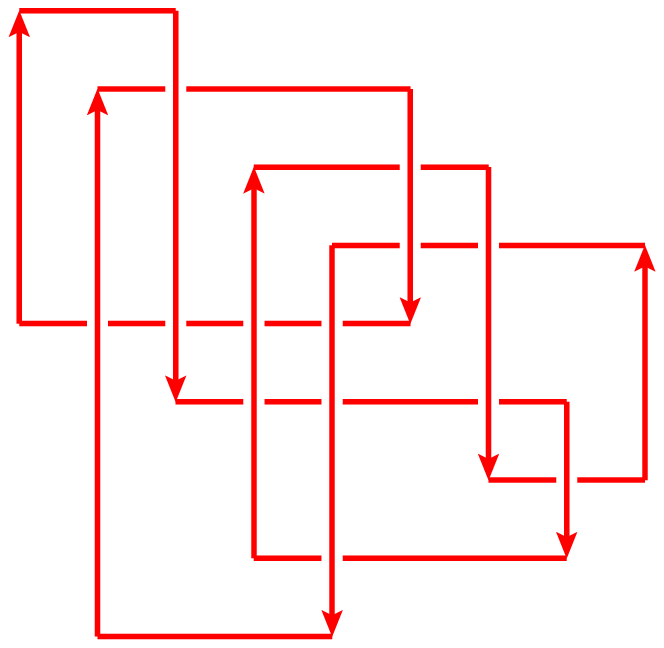} &
\rb{$\overline{1}\,2\,\overline{1}\,2\,3\,3\,\overline{2}\,1\,\overline{2}\,\overline{3}\,2$}
&
\multirow{2}{*}{\ding{55}} & \multirow{2}{*}{$\checkmark$} &
\rb{$\Aug(\HThat_0,\Z/3,2,1)=5$} \\
\cline{2-3} \cline{6-6}
& \includegraphics[height=0.5in]{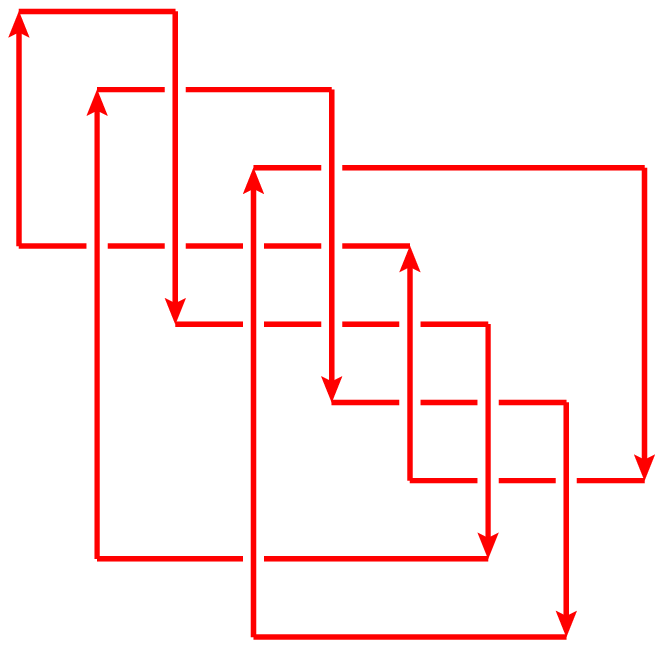} &
\rb{$\overline{2}\,3\,\overline{2}\,\overline{1}\,\overline{2}\,3\,\overline{2}\,1\,1\,1\,3$}
&&&
\rb{$\Aug(\HThat_0,\Z/3,2,1)=0$} \\
\hline
\multirow{2}{*}{$m(10_{140})$}
& \includegraphics[height=0.5in]{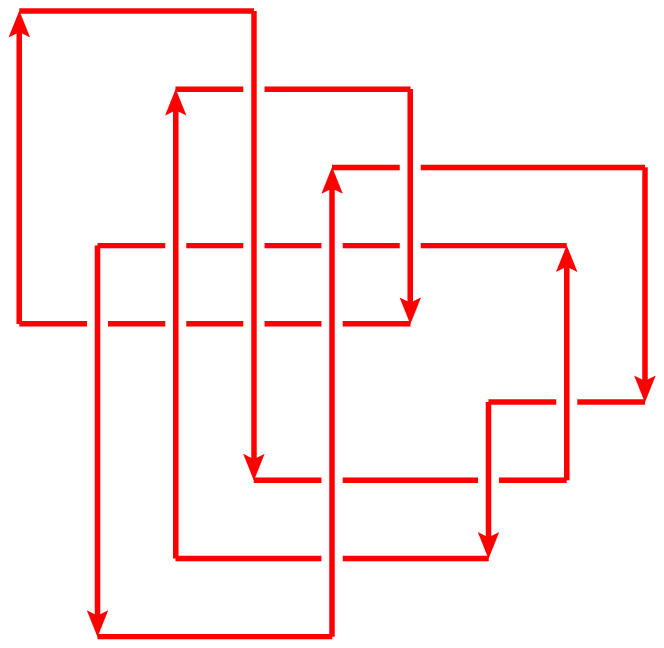} &
\rb{$1\,1\,\overline{2}\,1\,2\,\overline{1}\,\overline{1}\,\overline{3}\,2\,3\,3$}
&
$\checkmark$ & \multirow{2}{*}{$\checkmark$} &
\rb{$\Aug(\HThat_0,\Z/3,2,1)=1$} \\
\cline{2-3} \cline{6-6}
& \includegraphics[height=0.5in]{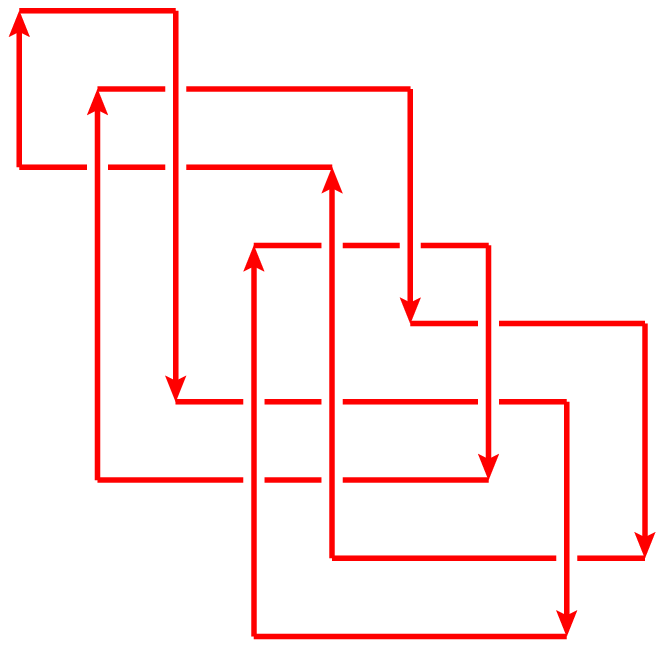} &
\rb{$1\,1\,\overline{2}\,1\,2\,\overline{1}\,\overline{1}\,3\,3\,2\,\overline{3}$}
& \raisebox{0.3in}{\cite{NOT}} &&
\rb{$\Aug(\HThat_0,\Z/3,2,1)=2$} \\
\hline
\multirow{2}{*}{$m(10_{145})$}
& \includegraphics[height=0.5in]{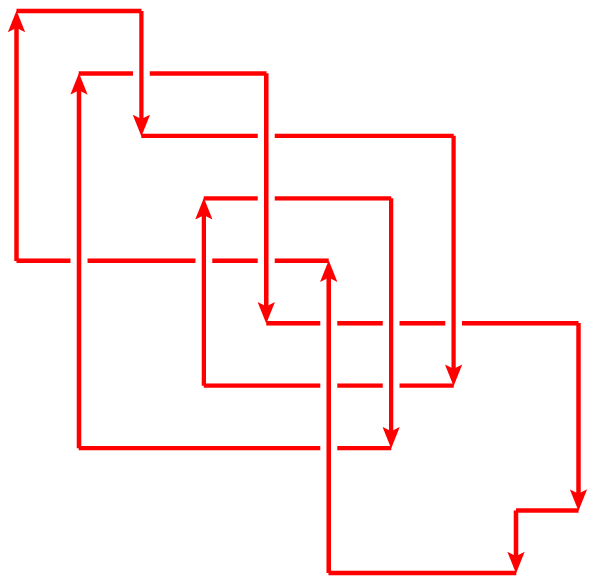} &
\rb{$\overline{2}\,3\,3\,2\,\overline{1}\,2\,1\,3\,2\,2\,1\,\overline{4}$}
&
\multirow{2}{*}{$\checkmark$ \cite{Atlas}} & \multirow{2}{*}{$\checkmark$} &
\rb{$\Aug(\HThat_0,\Z/3,1,1)=0$} \\
\cline{2-3} \cline{6-6}
& \includegraphics[height=0.5in]{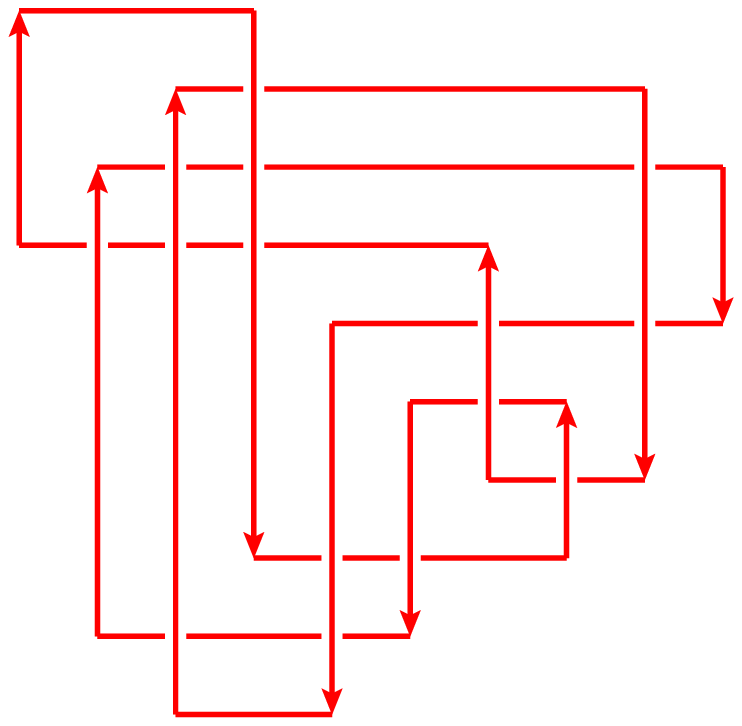} &
\rb{$3\,2\,1\,\overline{3}\,\overline{4}\,\overline{2}\,\overline{3}\,1\,2\,2\,1\,3\,4\,4$}
&&&
\rb{$\Aug(\HThat_0,\Z/3,1,1)=1$} \\
\hline
\multirow{2}{*}{$10_{160}$}
& \includegraphics[height=0.5in]{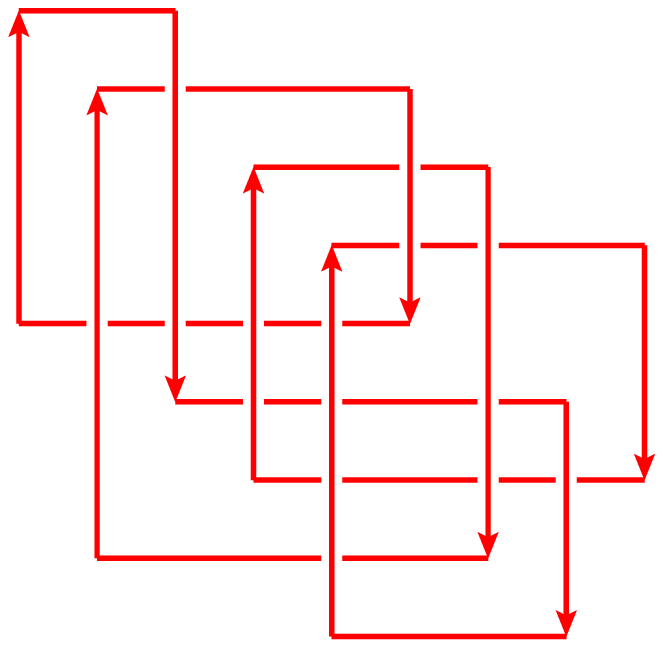} &
\rb{$\overline{2}\,3\,\overline{2}\,\overline{1}\,3\,2\,3\,2\,3\,1\,1$}
&
\multirow{2}{*}{\ding{55}} & \multirow{2}{*}{?} &
\rb{---} \\
\cline{2-3} \cline{6-6}
& \raisebox{0.25in}{\includegraphics[height=0.5in,angle=180]{10_160b.eps}} &
(mirror of previous braid)
&&&
---  \\
\hline
\multirow{2}{*}{$m(10_{161})$}
& \includegraphics[height=0.5in]{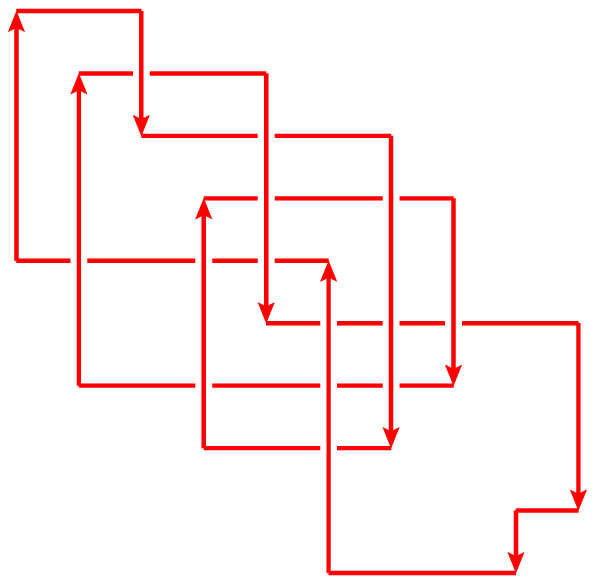} &
\rb{$\overline{1}\,2\,1\,1\,1\,2\,2\,1\,1\,2\,\overline{3}$}
&
\multirow{2}{*}{$\checkmark$ \cite{Atlas}} & \multirow{2}{*}{$\checkmark$} &
\rb{$\Aug(\HThat_0,\Z/3,1,1)=0$} \\
\cline{2-3} \cline{6-6}
& \includegraphics[height=0.5in]{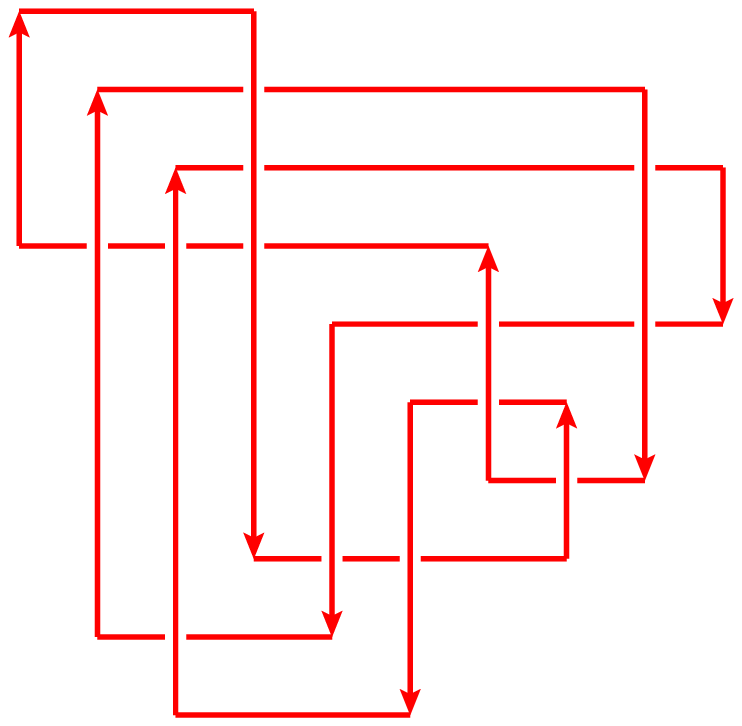} &
\rb{$2\,\overline{1}\,2\,2\,1\,3\,3\,2\,2\,2\,\overline{1}\,2\,\overline{3}$}
&&&
\rb{$\Aug(\HThat_0,\Z/3,1,1)=1$} \\
\hline
\multirow{2}{*}{$12n_{591}$}
& \includegraphics[height=0.5in]{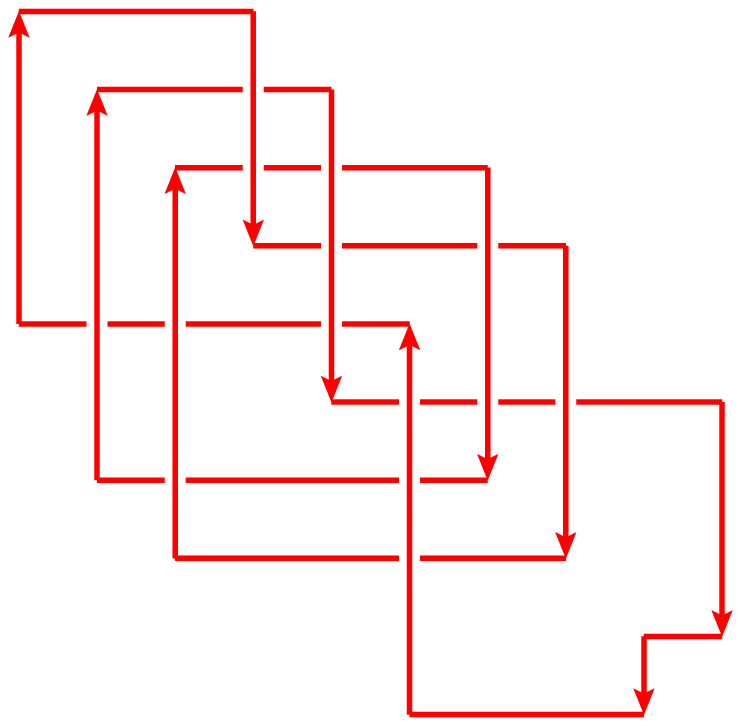} &
\rb{$3\,2\,3\,2\,\overline{1}\,3\,2\,1\,3\,2\,1\,2\,1\,\overline{4}$}
&
\multirow{2}{*}{$\checkmark$ \cite{Atlas}} & \multirow{2}{*}{$\checkmark$} &
\rb{$\Aug(\HThat_0,\Z/3,1,1)=0$} \\
\cline{2-3} \cline{6-6}
& \includegraphics[height=0.5in]{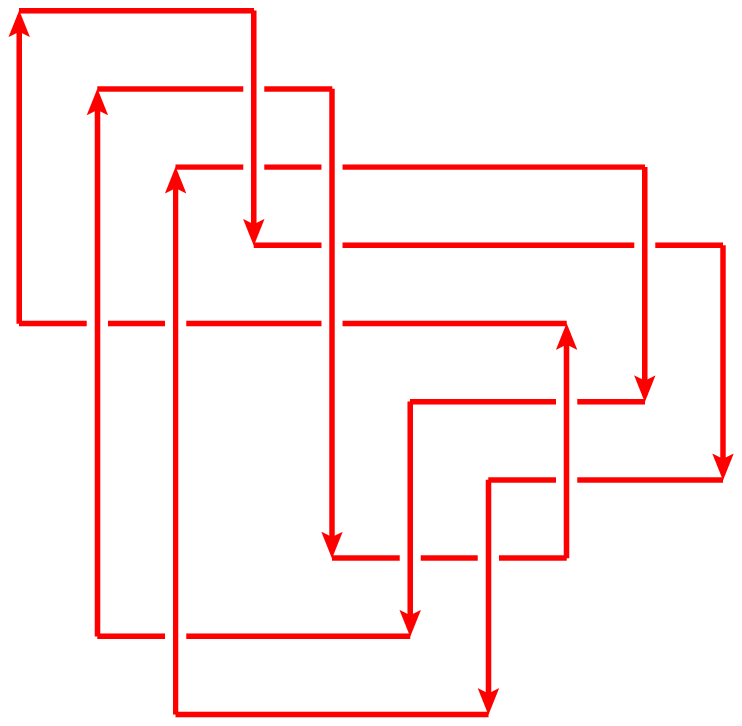} &
\rb{$\overline{2}\,\overline{3}\,\overline{1}\,\overline{2}\,4\,3\,4\,3\,2\,1\,2\,1\,2\,1\,4\,3\,4\,3$}
&&&
\rb{$\Aug(\HThat_0,\Z/3,1,1)=1$} \\
\hline
\end{tabular}
\end{center}

\newpage
\bibliographystyle{alpha}

\def\cprime{$'$}

\end{document}